\documentclass[a4paper]{plan-base}
\usepackage{tikz}
\usepackage{stmaryrd}
\usepackage[labelformat=simple]{subcaption}

\newlength{\w}

\usepackage[english]{babel}
\usepackage{caption}
\usepackage{subcaption}
\usepackage{comment}
\usepackage{footmisc}
\usepackage{bbm}
\usepackage{amsmath}
\usepackage{amssymb}
\usepackage{amsthm}
\usepackage{ifpdf}

\ifpdf
    \usepackage[pdftex]{hyperref}
\else
    \usepackage[hypertex]{hyperref}
\fi

\usepackage{color}
\definecolor{hrefcolor}{rgb}{0.0,0.5,0.8}
\definecolor{hlgreen}{rgb}{0,0.7,0}

\hypersetup{
   colorlinks=true,
   linkcolor=hrefcolor,
   citecolor=hlgreen,
   filecolor=hrefcolor,
   urlcolor=hrefcolor,
}

\ifdefined\tikz
\else
    \ifpdf
        \usepackage[pdftex]{graphicx}
    \else
        \usepackage{graphicx}
        
    \fi

\fi

\newenvironment{enumroman}
    {\renewcommand{\theenumi}{(\roman{enumi})}
     \renewcommand{\labelenumi}{(\roman{enumi})}
     \begin{enumerate}
        \setlength{\leftmargin}{3.0em}
        \setlength{\labelwidth}{2.5em}
        \setlength{\labelsep}{0.5em}
    }
    {\end{enumerate}}

\newcounter{remcount}

\newtheorem{theorem}{Theorem}
\newtheorem{corollary}{Corollary}
\newtheorem{lemma}{Lemma}
\newtheorem{proposition}{Proposition}

\theoremstyle{definition}
\newtheorem{definition}{Definition}
\newtheorem{assumption}{Assumption}
\newtheorem*{assumption*}{Assumption}
\newtheorem{remark}{Remark}
\newtheorem*{remark*}{Remark}
\newtheorem*{definition*}{Definition}

\numberwithin{equation}{section}
\numberwithin{lemma}{section}
\numberwithin{theorem}{section}
\numberwithin{proposition}{section}
\numberwithin{definition}{section}
\numberwithin{remark}{section}
\numberwithin{example}{section}
\numberwithin{assumption}{section}
\numberwithin{algorithm}{section}
\numberwithin{corollary}{section}

\newcommand*{\doi}[1]{doi:\href{http://dx.doi.org/#1}{\detokenize{#1}}}

\newcommand{\term}{\emph}

\newcommand{\field}[1]{\mathbb{#1}}

\newcommand{\Z}{\mathbb{Z}}

\newcommand{\R}{\field{R}}
\newcommand{\B}{B}

\newcommand{\norm}[1]{\|#1\|}

\newcommand{\abs}[1]{|#1|}

\newcommand{\inv}[1]{#1^{-1}}
\newcommand{\grad}[1]{\nabla #1}

\newcommand{\Union}\bigcup
\newcommand{\Isect}\bigcap
\newcommand{\union}\cup
\newcommand{\isect}\cap
\newcommand{\bigunion}\bigcup
\newcommand{\bigisect}\bigcap

\newcommand{\defeq}{:=}

\newcommand{\downto}{\searrow}

\DeclareMathOperator*{\lip}{lip}

\DeclareMathOperator{\ri}{ri}

\DeclareMathOperator{\closure}{cl}

\DeclareMathOperator{\diam}{diam}

\makeatletter
\def \uminus@sym{\setbox0=\hbox{$\cup$}\rlap{\hbox 
        to\wd0{\hss\raise0.5ex\hbox{$\scriptscriptstyle{-}$}\hss}}\box0}
    \def \uminus    {\mathrel{\uminus@sym}}
\makeatother

\newcommand{\mathvar}[1]{\textup{#1}}

\def\Xint#1{\mathchoice
{\XXint\displaystyle\textstyle{#1}}%
{\XXint\textstyle\scriptstyle{#1}}%
{\XXint\scriptstyle\scriptscriptstyle{#1}}%
{\XXint\scriptscriptstyle\scriptscriptstyle{#1}}%
\!\int}
\def\XXint#1#2#3{{\setbox0=\hbox{$#1{#2#3}{\int}$ }
\vcenter{\hbox{$#2#3$ }}\kern-.6\wd0}}

\def\dashint{\Xint-}

\renewcommand{\tilde}{\widetilde}
\renewcommand{\ae}{a.e.~}
\newcommand{\iprod}[2]{\langle #1,#2\rangle}

\DeclareMathOperator{\Sym}{Sym}
\newcommand{\Tensor}{\mathcal{T}}
\newcommand{\BDspace}{\mathvar{BD}}

\newcommand{\BVspace}{\mathvar{BV}}

\newcommand{\Eabs}{\mathcal{E}}
\newcommand{\Ejump}{E^j}
\newcommand{\Ecantor}{E^c}
\newcommand{\Meas}{\mathcal{M}}

\renewcommand{\H}{\mathcal{H}}
\renewcommand{\L}{\mathcal{L}}
\newcommand{\restrict}{\llcorner}
\newcommand{\BD}{\partial}
\renewcommand{\d}{\,d} 

\newcommand{\TGV}{\mathvar{TGV}}
\newcommand{\TGVns}{\mathvar{nsTGV}}
\newcommand{\TV}{\mathvar{TV}}

\newcommand{\nonsym}{\text{ns}}
\newcommand{\sym}{\text{s}}

\newcommand{\BALL}{V}
\DeclareMathOperator*{\esssup}{ess\,sup}

\DeclareMathOperator{\support}{supp}

\DeclareMathOperator{\divergence}{div}

\makeatletter
\def \weaktostar@sym{\setbox0=\hbox{$\rightharpoonup$}\rlap{\hbox 
        to\wd0{\hss\raise1ex\hbox{$\scriptscriptstyle{*\,}$}\hss}}\box0}
    \def \weaktostar    {\mathrel{\weaktostar@sym}}
\makeatother

\newcommand{\alphavec}{{\vec \alpha}}

\newcommand{\Uspace}{L^1(\Omega)}
\newcommand{\BVUspace}{\BVspace(\Omega)}
\newcommand{\EUspace}{\Meas(\Omega; \R^m)}
\newcommand{\Wspace}{L^1(\Omega; \R^m)}
\newcommand{\BVWspace}{\BVspace(\Omega; \R^m)}
\newcommand{\EWspace}{\Meas(\Omega; \Sym^2(\R^m))}
\newcommand{\DWspace}{\Meas(\Omega; \Tensor^2(\R^m))}
\newcommand{\jacobian}{\mathcal{J}}
\newcommand{\jacobianf}[3]{\mathcal{J}_{#1} #2(#3)}
\newcommand{\jacobianF}[2]{\mathcal{J}_{#1} #2}
\DeclareMathOperator{\DIFF}{d}

\renewcommand{\norm}[1]{\|#1\|}

\let\hatold=\hat
\let\tildeold=\tilde
\renewcommand{\hat}[1]{\mathchoice{\widehat{#1}}{\widehat{#1}}{\hatold#1}{\hatold#1}}
\renewcommand{\tilde}[1]{\mathchoice{\widetilde{#1}}{\widetilde{#1}}{\tildeold#1}{\tildeold#1}}

\newcommand{\Dfull}{D}
\newcommand{\Dabs}{\grad}
\newcommand{\Djump}{D^j}
\newcommand{\Dcantor}{D^c}
\newcommand{\ICTV}{\mathvar{ICTV}}
\DeclareMathOperator{\IC}{\square}
\newcommand{\allLambda}{\isect}

\newcommand{\pf}[1]{#1{}_{\#}}
\def\Xint#1{\mathchoice
{\XXint\displaystyle\textstyle{#1}}%
{\XXint\textstyle\scriptstyle{#1}}%
{\XXint\scriptstyle\scriptscriptstyle{#1}}%
{\XXint\scriptscriptstyle\scriptscriptstyle{#1}}%
\!\int}
\def\XXint#1#2#3{{\setbox0=\hbox{$#1{#2#3}{\int}$ }
\vcenter{\hbox{$#2#3$ }}\kern-.6\wd0}}

\def\dashint{\Xint-}
\newcommand{\bitransfull}[2]{T_{#1,#2}}
\newcommand{\bitrans}[2]{G_{#1,#2}}
\newcommand{\bitransjac}[2]{J_{#1,#2}}
\newcommand{\lipjac}[1]{A_{#1}}
\newcommand{\didtrans}[1]{D_{#1}}
\newcommand{\didtransjac}[1]{J_{#1}}
\newcommand{\ntdis}[1]{\overline M_{#1}}
\newcommand{\fst}[1]{\overline#1}
\newcommand{\snd}[1]{\underline#1}
\newcommand{\gammaone}{{\fst\gamma}}
\newcommand{\gammatwo}{{\snd\gamma}}
\newcommand{\uone}{{\overline u}}
\newcommand{\utwo}{{\underline u}}
\newcommand{\LipClass}{\mathcal{F}}
\newcommand{\gammax}[3]{\gamma_{#1,#3}}
\newcommand{\loc}{\mathrm{loc}}
\newcommand{\linear}{\mathcal{L}}
\newcommand{\RCa}{R^a}

\newcommand{\bvdiff}[2]{\hat{#1}_{#2}}

\newcommand{\partpf}[3]{\pf#1\llbracket#2,#3\rrbracket}

\begin{document}

\title{The jump set under geometric regularisation.\\
    \normalsize{Part 2: Higher-order approaches}
    }

\author{Tuomo Valkonen\thanks{
    Research Center on Mathematical Modeling,
    Escuela Politécnica Nacional de Quito, Ecuador.
    \\
    E-mail: \texttt{tuomo.valkonen@iki.fi}.
    }
    }

\maketitle

\begin{abstract}
In Part 1, we developed a new technique based 
on Lipschitz pushforwards for proving the jump set containment 
property $\H^{m-1}(J_u \setminus J_f)=0$ of solutions $u$ to total
variation denoising. We demonstrated that the technique
also applies to Huber-regularised $\TV$. Now, in this Part 2,
we extend the technique to higher-order regularisers. We
are not quite able to prove the property for total generalised
variation (TGV) based on the symmetrised gradient for the second-order
term. We show that the property holds under three conditions: First,
the solution $u$ is locally bounded. Second, the second-order variable
is of locally bounded variation, $w \in \BVspace_\loc(\Omega; \R^m)$, 
instead of just bounded  deformation, $w \in \BDspace(\Omega)$.
Third, $w$ does not jump on $J_u$ parallel to it.
The second condition can be achieved for non-symmetric TGV. 
Both the second and third condition can be achieved
if we change the Radon (or $L^1$) norm of the symmetrised gradient
$Ew$ into an $L^p$ norm, $p>1$, in which case Korn's inequality holds.
We also consider the application of the technique to
infimal convolution $\TV$, and study the limiting behaviour of
the singular part of $D u$, as the second parameter of $\TGV^2$ 
goes to zero. Unsurprisingly, it vanishes, but in numerical
discretisations the situation looks quite different.
Finally, our work additionally includes a result on
$\TGV$-strict approximation in $\BVspace(\Omega)$.

    
    
\end{abstract}

\section{Introduction}

We introduced in Part 1 \cite{tuomov-jumpset} the
\emph{double-Lipschitz comparability condition} 
of a regularisation functional $R$. Roughly
\begin{equation}
    \label{eq:double-lip}
    R(\pf\gammaone u) + R(\pf\gammatwo u) - 2R(u) 
    \le
    \bitransfull{\gammaone}{\gammatwo} \abs{Du}(\closure U),
\end{equation}
whenever $\gammaone,\gammatwo: \Omega \to \Omega$ are bi-Lipschitz 
transformations reducing to the identity outside $U \subset \Omega$.
Constructing specific Lipschitz shift transformations around
a point $x \in J_u$, for which the constant
$\bitransfull{\gammaone}{\gammatwo}=O(\rho^2)$ for $\rho>0$ the 
size of the shift, we were able to prove the jump set containment
\begin{equation}
    \label{eq:no-jumps}
    \tag{J}
    \H^{m-1}(J_u \setminus J_f)=0
\end{equation}
for $u \in \BVspace(\Omega)$ the solution of the denoising or 
regularisation problem
\begin{equation}
    \label{eq:prob}
    \tag{P}
    \min_{u \in \BVspace(\Omega)} 
        \int_\Omega \phi(\abs{f(x)-u(x)}) \d x + R(u).
\end{equation}
The admissible fidelities $\phi$ include here $\phi(t)=t^p$ 
for $1 < p < \infty$. For $p=1$ we produced somewhat weaker results
comparable to those for total variation ($\TV$) in \cite{duval2009tvl1}.
The admissible regularisers $R$ included, obviously, total
variation, for which the property was already proved previously
by level set techniques \cite{caselles2008discontinuity}. We also
showed the property for Huber-regularised total variation as a 
new contribution besides the technique. If we non-convex total
variation models and the Perona-Malik anisotropic diffusion
were well-posed, we demonstrated that the technique would also
apply to them.

The development of the new technique was motivated by higher-order
regularisers, in particular by total generalised variation ($\TGV$,
\cite{bredies2009tgv}), for which the level set technique is not
available due to the lack of a co-area formula. In this Part 2,
we now aim to extend our Lipschitz pushforward technique to
variants of $\TGV$ as well as infimal convolution TV
($\ICTV$, \cite{chambolle97image}). In order to do this, we need
to modify the double-Lipschitz comparability criterion 
\eqref{eq:double-lip} a little bit. Namely, we will in
Section \ref{sec:problem} introduce rigorously a 
\emph{partial double-Lipschitz comparability condition}
of the form
\begin{equation}
    \label{eq:double-lip-shift}
    R(\pf\gammaone(u - v) + v) + R(\pf\gammatwo(u-v) + v) - 2R(u) 
    \le
    \bitransfull{\gammaone}{\gammatwo} \abs{D(u-v)}(\closure U)
    +\text{small terms}.
\end{equation}
Here, in comparison to \eqref{eq:double-lip}, we have subtracted
$v$ from $u$ before the pushforward. The idea is the
same as in the application the jump set containment result
for $\TV$ to prove it for $\ICTV$. Namely, as we may recall
\[
    \ICTV_{\alphavec}(u) \defeq
        \min_{v \in W^{1,1}(\Omega), \grad v \in \BVspace(\Omega; \R^m)}
        \alpha \norm{Du-\grad v}_{2,\Meas(\Omega; \R^m)}
        + \beta \norm{D\grad v}_{F,\Meas(\Omega; \R^{m \times m})},
\]        
where $\alphavec=(\beta,\alpha)$.
Now, if $u$ solves \eqref{eq:prob} for $R=\ICTV_{\alphavec}$, then
$u$ solves
\[
    \min_{u \in \BVspace(\Omega)}
        \int_\Omega \phi(\abs{f(x)-u(x)}) \d x + \norm{Du-\grad v}_{2,\Meas(\Omega; \R^m)}.
\]
with $v$ fixed. Otherwise written, $\bar u=u-v$ solves for $\bar f=f-v$
the total variation denoising problem
\[
    \min_{u \in \BVspace(\Omega)}
        \int_\Omega \phi(\abs{\bar f(x)-\bar u(x)}) \d x + \norm{D\bar u}_{2,\Meas(\Omega; \R^m)}.
\]
Since $v \in W^{1,1}(\Omega)$ has no jumps, $J_{\bar f}=J_f$, 
the fact \eqref{eq:no-jumps} that $\ICTV$ introduces no jumps 
follows from the corresponding result for $\TV$. 

The idea with $v$ in \eqref{eq:double-lip-shift} is roughly the 
same as this: to remove the second-order information from the problem,
and reduce it into a first-order one. However, unlike in the case 
of $\ICTV$, generally, we cannot reduce the problem to $\TV$. 
Indeed, written in the differentiation cascade formulation 
\cite{sampta2011tgv}, second-order $\TGV$ reads as
\begin{equation}
    \label{eq:tgv2-def}
    \TGV^2_{\alphavec}(u) \defeq
        \min_{w \in \BDspace(\Omega)}
        \alpha \norm{Du-w}_{2,\Meas(\Omega; \R^m)}
        + \beta \norm{Ew}_{F,\Meas(\Omega; \Sym^2(\R^m)}.
\end{equation}
Here $\BDspace(\Omega)$ is the space of vector fields of 
\emph{bounded deformation} on $\Omega$, and $Ew$ the
symmetrised gradient. Now, we do not generally have
$w=\grad v$ for any function $v$, which is the reason
that the analysis is not as simple as that of $\ICTV$.
Standard $\TGV^2$ is also significantly complicated 
by the symmetrised gradient $Ew$, and we cannot obtain
as strong results for it, our results depending on
assumptions on $w$. Namely, we need that $w$
is ``BV-differentiable'', or, in practise that
$w \in \BVspace_\loc(\Omega; \R^m)$ instead of just
$w \in \BDspace(\Omega)$, and that the projection
$P_{z_\Gamma}^\perp(w^+(x)-w^-(x))=0$, on a Lipschitz
graph $\Gamma$, representing $J_u$, parametrised on the 
plane orthogonal to $z_\Gamma \in \R^m$. 
These complications make us firstly consider the non-symmetric 
variant of $\TGV^2$, $\TGVns^2$, where $Ew$ in \eqref{eq:tgv2-def}
is replaced by $Dw$ and $w \in \BVspace(\Omega; \R^m)$.
Secondly, we consider variants of $\TGV^2$ employing 
for the second-order term $L^q$ energies, $q>1$. These
have the advantage that Korn's inequality holds. 
For all of these variants, and for ICTV, we obtain
stronger results than for $\TGV^2$ itself.

Our analysis of the specific regularisation functionals
is in Section \ref{sec:reg} after we study local
approximability of $w \in \BDspace(\Omega)$ and
approximability in terms of $\TGV$-strict convergence
in Section \ref{sec:approx}.
The analysis of the fidelity term $\int_\Omega \phi(\abs{f(x)-u(x)}) \d x$
is unchanged from Part 1 \cite{tuomov-jumpset}, and therefore
the main lemma is only cited in Section \ref{sec:problem},
where we state our assumptions on $R$ and $\phi$,
and prove \eqref{eq:no-jumps} for \eqref{eq:prob}
by combining the separate estimates for the fidelity and regularity terms.
We concentrate on $p$-increasing fidelities for $1 < p < \infty$.
The case $p=1$ from Part 2 could also be extended, but we have
chosen to concentrate on the case $p > 1$ where stronger results
exist. As an addendum to this qualitative study, we also study 
quantitatively in Section \ref{sec:lptgv2-bounds}
the limiting behaviour of the singular part $D^s u$ of $D u$
for $\TGV^2$ as $\beta \downto 0$. The behaviour is quite surprising, 
as on the discrete scale $\TGV^2$ appears to preserve jumps in 
the limit, but analysis shows that the jumps disappear.

The class of problems \eqref{eq:prob} is of importance, in particular, for 
image denoising. We wish to know the structure of $J_u$ in order to see
that the denoising problem does not introduce undesirable artefacts,
new edges, which in images model different materials and depth information.
Higher-order geometric and other recently introduced image regularisers
such a $\TGV$ \cite{bredies2009tgv}, $\ICTV$ \cite{chambolle97image}, 
Euler's elastica \cite{chan2002euler,shen2003euler}, and many
other variants \cite{lysaker2003noise,burger2012regularized,
papafitsoros2012combined,chan2000high,DFLM2009,didas2009properties,
bertozzi2004low} are, in fact, motivated by one serious artefact
of the conventional total variation regulariser. This is the 
stair-casing effect. Further, non-convex total variation schemes
and ``lower-order fidelities'' such as Meyer's G-norm and the 
Kantorovich-Rubinstein norm, have recently received increased 
attention in an attempt to, respectively, better model real image
gradient statistics \cite{huang1999statistics,HiWu13_siims,HiWu14_coap,ochsiterated,tuomov-tvq}
or texture \cite{meyer2002oscillating,vese2003modelingtextures,tuomov-krtv}.
Very little is known about any of these analytically. For $\TGV^2$
we primarily have the results on one-dimensional domains in
\cite{l1tgv,papafitsoros2013study}. We hope
that our work in this pair of papers provides an impetus and
roots for a technique for the study of many of these and future
approaches. We begin our study after going through the obligatory
preliminaries in the following Section \ref{sec:prelim}.
We finish the study with a few final words in Section \ref{sec:conclusion}.

\section{Notations and useful facts}
\label{sec:prelim}

We begin by introducing the tools necessary for our work.
Much of this material is the same as in Part 1 \cite{tuomov-jumpset}; 
we have however decided to make this manuscript to be mostly self-contained,
legible without having to delve into the extensively detailed 
analysis of Part 1.
We will also include additional information on tensor fields and 
functions of bounded deformation, $\BDspace(\Omega)$. These are
crucial for the  definition of $\TGV$. First we introduce basic 
notations for sets, mappings, measures, and tensors. We then move
on to tensor fields and Lipschitz mappings and graphs. Finally, 
we discuss distributional gradients of tensor fields, which allow
us to define bounded variation and deformation in a unified
way.

\subsection{Basic notations}

We denote by $\{e_1,\ldots,e_m\}$ the standard basis of $\R^m$.
The boundary of a set $A$ we denote by $\BD A$, and the 
closure by $\closure A$.  The $\{0,1\}$-valued indicator function
we write as $\chi_A$.
We denote the open ball of radius $\rho$ centred at $x \in \R^m$ 
by $\B(x, \rho)$. We denote by $\omega_m$ the volume of the unit 
ball $\B(0, 1)$ in $\R^m$.

For $z \in \R^m$, we denote by
$z^\perp \defeq \{ x \in \R^m \mid \iprod{z}{x}=0\}$
the hyperplane orthogonal to $\nu$ , whereas
$P_z$ denotes the projection operator onto the subspace
spanned by $z$, and $P_z^\perp$ the projection onto 
$z^\perp$. If $A \subset z^\perp$, we denote by
$\ri A$ the \emph{relative interior} of $A$ in
$z^\perp$ as a subset of $\R^m$.

Let $\Omega \subset \R^m$ be an open set.
We then denote the space of (signed) 
Radon measures on $\Omega$ by $\Meas(\Omega)$. 
If $V$ is a vector space, then the space of
Radon measures on $\Omega$ with values in $V$ 
is denoted $\Meas(\Omega; V)$.
The $k$-dimensional Hausdorff measure,
on any given ambient space $\R^m$, ($k \le m$), 
is denoted by $\H^k$,
while $\L^m$ denotes the Lebesgue measure on $\R^m$.
The total variation (Radon) norm of a measure $\mu$ is denoted 
$\norm{\mu}_{\Meas(\R^m)}$. 

For vector-valued measures
$\mu \in \Meas(\Omega; \R^k)$, we use the notation
$\norm{\mu}_{q, \Meas(\R^m)}$ to indicate that the
finite-dimensional base norm is the $q$-norm. We
use the same notation for vector fields $w \in L^p(\Omega; \R^k)$,
namely
\[
    \norm{w}_{q, L^p(\Omega)} \defeq \left(\int_\Omega \norm{w(x)}_q^p \d x \right)^{1/p}.
\]    

For a measurable set $A$, we  denote by $\mu \restrict A$ the 
restricted measure defined
by $(\mu \restrict A)(B) \defeq \mu(A \isect B)$.
The notation $\mu \ll \nu$ means that $\mu$ is absolutely
continuous with respect to the measure $\nu$, and $\mu \perp \nu$ 
that $\mu$ and $\nu$ are mutually singular. The singular and
absolutely continuous (with respect to the Lebesgue measure)
part of $\mu$ are denoted $\mu^a$ and $\mu^s$, respectively.

We denote the $k$-dimensional upper resp.~lower density
of $\mu$ by
\[
    \Theta^*_k(\mu; x) \defeq \limsup_{\rho \downto 0} \frac{\mu(\B(x, \rho))}{\omega_k \rho^k},
    \quad
    \text{resp.}
    \quad
    \Theta_{*,k}(\mu; x) \defeq \liminf_{\rho \downto 0} \frac{\mu(\B(x, \rho))}{\omega_k \rho^k}.
\]
The common value, if it exists, we denote by $\Theta_k(\mu; x)$. 

Finally, we often denote by $C$, $C'$, $C'''$
arbitrary positive constants,
and use the plus-minus notation 
$a^\pm=b^\pm$ in to mean that both $a^+=b^+$ and $a^-=b^-$ hold.

\subsection{Lipschitz and $C^1$ graphs}

A set $\Gamma \subset \R^m$ is called a Lipschitz ($m-1$)-graph
(of Lipschitz factor $L$), if there exist a unit vector $z_\Gamma$,
an open set $V_\Gamma \subset z_\Gamma^\perp$, and a Lipschitz map 
$f_\Gamma: V_\Gamma \to \R$, of Lipschitz factor at most $L$,
such that
\[
    \Gamma = \{ v+f_\Gamma(v)z_\Gamma \mid v \in V_\Gamma \}.
\]
If $f_\Gamma \in C^1(V_\Gamma)$, we cal $\Gamma$ a $C^1$ ($m-1$)-graph.
We also define $g_\Gamma: V_\Gamma \to \R^m$ by
\[
    g_\Gamma(v)=v+z_\Gamma f_\Gamma(v).
\]
Then
\[
    \Gamma = g_\Gamma(V_\Gamma).
\]

We denote the open domains ``above'' and ``beneath''
$\Gamma$, respectively, by
\[
    \Gamma^+ \defeq \Gamma+(0,\infty)z_\Gamma,
    \quad
    \text{and}
    \quad
    \Gamma^- \defeq \Gamma+(-\infty, 0)z_\Gamma.
\]
We recall
that by Kirszbraun's theorem, we may extend 
the domain of $g_\Gamma$ from $V_\Gamma$
to the whole space $z_\Gamma^\perp$ without
altering the Lipschitz constant. Then
$\Gamma$ splits $\Omega$ into the two open halves
$\Gamma^+ \isect \Omega$ and $\Gamma^- \isect \Omega$.
We often use this fact.

\subsection{Mappings from a subspace}

We denote by $\linear(V; W)$ the space of linear maps between
the vector spaces $V$ and $W$.
If $L \in \linear(V; \R^k)$, where $V \sim \R^n$, ($n \le k$),
is a finite-dimensional Hilbert space,
Then $L^* \in \linear(\R^k; V^*)$ denotes the adjoint,
and the $n$-dimensional Jacobian is defined as
\cite{ambrosio2000fbv}
\[
    \jacobian_{n}[L] \defeq \sqrt{\det (L^* \circ L)}.
\]
With the gradient of a Lipschitz function $f: V \to \R^k$
defined in ``components as columns order'', 
$\grad f(x) \in \linear(\R^k; V)$, we extend this 
notation for brevity as
\[
    \jacobianf{n}{f}{x} \defeq \jacobian_{n}[(\grad f(x))^*].
\]
If $\Omega \subset V$ is a measurable set, 
and $g \in L^1(\Omega)$, the \term{area formula} 
may then be stated 
\begin{equation}
    \label{eq:areaformula}
    \int_{\R^k} \sum_{x \in \Omega \isect \inv f(y)} g(x) \d\H^{n}(y)
    = \int_\Omega g(x) \jacobianf{n}{f}{x} \d \H^n(x).
\end{equation}
That this indeed holds in our sitting of finite-dimensional Hilbert
spaces $V \sim \R^n$ follows by a simple argument from the area formula
for $f: \R^n \to \R^k$, stated in, e.g, \cite{ambrosio2000fbv}. 
We only use the cases $V = z^\perp$ for some $z \in \R^m$ ($n=m-1$),
or $V=\R^m$ ($n=m$).

We also denote by
\[
    C^{2,\allLambda}(V) \defeq \Isect_{\lambda \in (0,1)} C^{2,\lambda}(V)
\]
the class of functions that are twice differentiable (as defined above
for tensor fields) with a Hölder continuous second differential for
all exponents $\lambda \in (0,1)$.

The Lipschitz factor of a Lipschitz mapping $f$ we denote by $\lip f$.
We also recall that a Lipschitz transformation $\gamma: U \to V$ with
$U, V \subset \R^m$ has the \term{Lusin $N$-property} if it maps
$\L^m$-negligible sets to $\L^m$-negligible sets.

If $\gamma: \Omega \to \Omega$ is a 1-to-1 Lipschitz transformation, 
and $u: \Omega \to \Omega$ a Borel function, we define the 
pushforward $u_\gamma \defeq \pf\gamma u \defeq u \circ \gamma^{-1}$.
Finally, we denote the identity transformation by $\iota(x)=x$.

\subsection{Tensors and tensor fields}

We now introduce tensors and tensor fields. We simplify the
treatment from its full differential-geometric setting, as can
be found in, e.g., \cite{bishop1980tensor}, as we are working 
on finite-dimensional Hilbert spaces.
These definitions and our approach to defining $\TGV^2$ follow 
that in \cite{tuomov-dtireg}.

We let $V_1,\ldots,V_k$ be finite-dimensional Hilbert spaces,
$V_j \sim \R^{m_j}$ with corresponding bases
$\{e_1^j,\ldots,e_{m_j}^j\}$, ($j=1,\ldots,k$).
A $k$-tensor is then a $k$-linear mapping 
$A: V_1 \times \cdots \times V_k \to \R$. 
We denote $A \in \Tensor(V_1,\ldots,V_k)$.
If $V_j=V$ for all $j=1,\ldots,k$, we write
$\Tensor^k(V) \defeq \Tensor(V_1,\ldots,V_k)$.
A symmetric tensor $A \in \Sym^k(V) \subset \Tensor^k(V)$ satisfies
for any permutation $\pi$ of $\{1,\ldots,k\}$ and
any $c_1,\ldots,c_k \in V$ that
$A(c_{\pi 1},\ldots,c_{\pi k})=A(c_1,\ldots,c_k)$,
For conciseness of notation, we often identify $V \sim \Tensor^1(V)$ 
through the mapping $V(x) = \iprod{V}{x}$.

For a $A \in \Tensor(V_1,\ldots,V_k)$
and $B \in \Tensor(V_{k+1},\ldots,V_{k+m})$
we define the $(m+k)$-tensor $A \otimes B \in \Tensor(V_1,\ldots,V_{k+m})$ 
by
\[
    (A \otimes B)(c_1,\ldots,c_{k+m})
    =A(c_1,\ldots,c_k)B(c_{k+1},\ldots,c_{k+m}).
\]

We define on $A, B \in \Tensor(V_1,\ldots,V_k)$ the inner product
\[
    \iprod{A}{B} \defeq \sum_{p_1=1}^{m_1} \cdots \sum_{p_k=1}^{m_k}
            A(e_{p_1}^1,\ldots,e_{p_k}^k)
            B(e_{p_1}^1,\ldots,e_{p_k}^k),
\]
and the Frobenius norm
\[
    \norm{A}_F \defeq \sqrt{\iprod{A}{A}}.
\]
If $k=1$ ,we simply denote $\norm{A} \defeq \norm{A}_2 \defeq \norm{A}_F$,
as the Frobenius norm agrees with the Euclidean norm.

Let then $u: \Omega \to \Tensor(V_1, \ldots, V_k)$ be a Lebesgue-measurable 
function on the domain $\Omega \subset V_0$, where $V_0 \sim \R^m$ is 
also a finite-dimensional Hilbert space. We define the norms
\[
    \norm{u}_{F,p}
    \defeq
    \Bigl(\int_\Omega \norm{u(x)}_F^p \d x\Bigr)^{1/p}\hspace{-3ex},
    \hspace{4ex} (p \in [1,\infty)),
    \quad
    \text{ and}
    \quad
    \norm{u}_{F,\infty}
    \defeq
        \esssup_{x \in \Omega} \norm{u(x)}_F,
\]
and the spaces
\begin{align}
    \notag
    L^p(\Omega; \Tensor(V_1,\ldots,V_k))
    &
    =\{u: \Omega \to \Tensor(V_1,\ldots,V_k) \mid u \text{ Borel},\, \norm{u}_{F,p} < \infty\},
    \quad
    (p \in [1,\infty]).
\end{align}
The spaces $L^p(\Omega; \Tensor^k(V))$ and $L^p(\Omega; \Sym^k(V))$ 
are defined analogously.


\subsection{Distributional gradients and tensor-valued measures}

For the definition of total generalised variation (TGV), 
we need to define the concept of a tensor-valued measure,
as well as the distributional differential $Du$ and the symmetrised
distributional $Eu$ on tensor fields. This is done now.
If the reader is satisfied with a cursory understanding of TGV,
this subsection may be skipped.

We start with tensor field divergences.
Let $u \in C^1(\Omega; \Tensor(V_1,\ldots,V_k))$, ($k\ge 0$).
The (Fr{\'e}chet) differential 
$\DIFF f(x) \in \Tensor(V_0, V_1,\ldots,V_k)$ at $x \in \Omega$ 
is defined by the limit
\[
    \lim_{h \to 0} \frac{\norm{f(x+h)-f(x)-\DIFF f(x)(h,\cdot,\ldots,\cdot)}_F}{\norm{h}_F} = 0.
\]
If $k \ge 1$, if $V_0=V_1$, we define the divergence, 
$\divergence u \in C(\Omega; \Tensor(V_2,\ldots,V_k))$ by contraction as
\[
    [\divergence u(x)](c_2,\ldots,c_k)
    \defeq
    \sum_{i=1}^{m_1}
    \DIFF u(x)(\xi_i^1,\xi_i^1, c_2, \ldots,c_k).
\]
Observe that if $u$ is symmetric, then so is $\divergence u$.
Moreover Green's identity
\[
    \int_{\Omega} \iprod{\DIFF u(x)}{\phi(x)} \d x
    =
    \int_{\Omega} \iprod{u(x)}{- \divergence \phi(x)} \d x
\]
holds for  $u \in C^1(\Omega; \Tensor(V_2,\ldots,V_k))$
and $\phi \in C^1_0(\Omega; \Tensor(V_1,\ldots,V_k))$
with $\Omega \subset V_1=V_0$.

Denoting by $X^*$ the continuous linear functionals
on the topological space $X$, we now define 
the distributional gradient
\[
    D u \in [C_c^\infty(\Omega; \Tensor^{k+1}(\R^m))]^*
\]
of $u \in L^1(\Omega; \Tensor^k(\R^m))$ by
\[
    D u(\varphi) \defeq -\int_\Omega \iprod{u(x)}{\divergence \varphi(x)} \d x,
    \quad
    (\varphi \in C_c^\infty(\Omega; \Tensor^{k+1}(\R^m))).
\]
Likewise we define the symmetrised distributional gradient
\[
    E u \in [C_c^\infty(\Omega; \Sym^{k+1}(\R^m))]^*
\]
of $u \in L^1(\Omega; \Tensor^k(\R^m))$ by
\[
    E u(\varphi) \defeq -\int_\Omega \iprod{u(x)}{\divergence \varphi(x)} \d x,
    \quad
    (\varphi \in C_c^\infty(\Omega; \Sym^{k+1}(\R^m))).
\]
We also define the ``Frobenius unit ball''
\[
    \BALL^{k}_{F,\nonsym} \defeq \{ \varphi \in C_c^\infty(\Omega; \Tensor^{k}(\R^m)) \mid \norm{\varphi}_{F,\infty} \le 1 \}.
\]
and the ``symmetric Frobenius unit ball''
\[
    \BALL^{k}_{F,\sym} \defeq \{ \varphi \in C_c^\infty(\Omega; \Sym^{k}(\R^m)) \mid \norm{\varphi}_{F,\infty} \le 1 \}.
\]
For our purposes it then suffices to define a tensor-valued 
measure $\mu \in \Meas(\Omega; \Tensor^{k}(\R^m))$ as a linear functional
$\mu \in [C_c^\infty(\Omega; \Tensor^{k}(\R^m))]^*$
bounded in the sense that the total variation norm
\[
    \norm{\mu}_{F,\Meas(\Omega; \Tensor^{k}(\R^m))} \defeq \sup\{\mu(\varphi) \mid  \varphi \in \BALL^{k}_{F,\nonsym}\} < \infty.
\]
For a justification of this definition, we refer to \cite{federer1969gmt}.
The definition of a symmetric measure $\mu \in \Meas(\Omega; \Sym^{k}(\R^m))$
is analogous with $\mu \in [C_c^\infty(\Omega; \Sym^{k}(\R^m))]^*$ and
\[
    \norm{\mu}_{F,\Meas(\Omega; \Sym^{k}(\R^m))} \defeq \sup\{\mu(\varphi) \mid  \varphi \in \BALL^{k}_{F,\sym}\} < \infty.
\]
It follows that $Du$ and $Eu$ are measures when they are bounded
on $\BALL^{k}_{F,\nonsym}$ and $\BALL^{k}_{F,\sym}$, respectively.
Observe that for $k=0$, it holds
$\Meas(\Omega; \Tensor^{0}(\R^m))=\Meas(\Omega; \Sym^{0}(\R^m))
=\Meas(\Omega)$, and for $k=1$, it holds
\[
    \Meas(\Omega; \Tensor^{1}(\R^m))=\Meas(\Omega; \Sym^{1}(\R^m))
    =: \EUspace.
\]

\begin{remark}
The choice of the Frobenius norm as the
finite-dimensional norm in the above definitions, indicated by
the subscript $F$, ensures isotropy and a degree of
rotational invariance for tensor fields. Some alternative
rotationally invariant norms, generalising the nuclear and the spectral 
norm  for matrices, are discussed \cite{tuomov-dtireg}.
\end{remark}

\subsection{Functions of bounded variation}

We say that a function $u: \Omega \to \R$ on a bounded 
open set $\Omega \subset \R^m$, is
of \term{bounded variation} (see, e.g., \cite{ambrosio2000fbv}
for a more thorough introduction),
denoted $u \in \BVspace(\Omega)$, if $u \in L^1(\Omega)$, 
and the distributional gradient $\Dfull u$ is a Radon measure.
Given a sequence $\{u^i\}_{i=1}^\infty \subset \BVspace(\Omega)$, 
weak* convergence is defined as $u^i \to u$ strongly in $L^1(\Omega)$
along with $\Dfull u^i \weaktostar \Dfull u$ weakly* in 
$\Meas(\Omega)$. The sequence converges \emph{strictly} if,
in addition to this, $\abs{Du^i}(\Omega) \to \abs{Du}(\Omega)$.

We denote by $S_u$ the approximate discontinuity set, i.e.,
the complement of the set where the Lebesgue limit $\tilde u$ exists. 
The latter is defined by
\[
    \lim_{\rho \downto 0}  \frac{1}{\rho^m}
        \int_{\B(x, \rho)} \norm{\tilde u(x) - u(y)} \d y = 0.
\]

The distributional gradient can
be decomposed as $\Dfull u = \Dabs u \L^m + \Djump u + \Dcantor u$, where
the density $\Dabs u$ of the \term{absolutely continuous part} of
$\Dfull u$ equals (a.e.) the approximate differential of $u$.
We also define the \term{singular part} as $D^s u=D^j u + D^c u$.
The \term{jump part} $\Djump u$ may be represented as
\begin{equation}
    \notag
    \Djump u = (u^+ - u^-) \otimes \nu_{J_u} \H^{m-1} \restrict J_u,
\end{equation}
where $x$ is in the \term{jump set} $J_u \subset S_u$ of $u$ if for some 
$\nu \defeq \nu_{J_u}(x)$ there exist two \emph{distinct} one-sided
traces $u^+(x)$ and $u^-(x)$, defined as satisfying
\begin{equation}
    \notag
    \lim_{\rho \downto 0}  \frac{1}{\rho^m}
        \int_{\B^\pm(x, \rho, \nu)} \norm{u^\pm(x) - u(y)} \d y= 0,
\end{equation}
where
$\B^\pm(x, \rho, \nu) := \{ y \in \B(x,\rho) \mid \pm \iprod{y-x}{\nu} \ge 0\}$.
It turns out that $J_u$ is countably $\H^{m-1}$-rectifiable and $\nu$ is
(a.e.) the normal to $J_u$. This former means that
there exist Lipschitz $(m-1)$-graphs $\{\Gamma_i\}_{i=1}^\infty$ such
that $\H^{m-1}(J_u \setminus \Union_{i=1}^\infty \Gamma_i)=0$.
Moreover, we have $\H^{m-1}(S_u \setminus J_u)=0$. 
The remaining \term{Cantor part} $\Dcantor u$ vanishes
on any Borel set $\sigma$-finite with respect to $\H^{m-1}$.

We will depend on the following basic properties of
densities of $Du$; for the proof, see, e.g.,
\cite[Proposition 3.92]{ambrosio2000fbv}.

\begin{proposition}
    \label{prop:bv-density}
    Let $u \in \BVspace(\Omega)$ for an open domain $\Omega \subset \R^m$.
    Define
    \[
        \tilde S_u \defeq \{x \in \Omega \mid \Theta_{*,m}(\abs{Du}; x)=\infty\},
        \quad
        \text{and}
        \quad
        \tilde J_u \defeq \{x \in \Omega \mid \Theta_{*,m-1}(\abs{Du}; x) >0\}.
    \]
    Then the following decomposition holds.
    \begin{enumroman}
        \item $\Dabs u = D u \restrict (\Omega \setminus \tilde S_u)$.
        \item $\Djump u = D u \restrict \tilde J_u$, 
            precisely $\tilde J_u \supset J_u$, and $\H^{m-1}(\tilde J_u \setminus J_u)=0$.
        \item $\Dcantor u = Du \restrict (\tilde S_u \setminus \tilde J_u)$.
    \end{enumroman}
\end{proposition}

We will require the following property of the traces along a Lipschitz
graph $\Gamma$.

\begin{lemma}[Part 1]
    \label{lemma:z-u}
    Let $u \in \BVspace(\Omega)$. Then there exists a Borel set $Z_u$
    with $\H^{m-1}(Z_u)=0$ such that
    every $x \in J_u \setminus Z_u$ is a Lebesgue
    point of the one-sided traces $u^\pm$, and
    \[
        \Theta^*_{m-1}(\abs{Du} \restrict (\Gamma^x)^+; x) = 0,
        \text{ and }
        \Theta^*_{m-1}(\abs{Du} \restrict (\Gamma^x)^-; x) = 0
    \]
    for a Lipschitz $(m-1)$-graph $\Gamma^x$, which satisfies
    the following.
    Firstly
    \[
        V_{\Gamma^x} \supset \B(P_{z_\Gamma}^\perp x, r(x))
    \]
    for some $r(x)>0$. Secondly the traces of $u$ at $x$ exist
    from both sides of $\Gamma^x$ and agree with $u^\pm(x)$.
\end{lemma}

\subsection{Functions of bounded deformation}

Similarly to the definition of a function
of bounded variation, a function $w \in L^1(\Omega; \R^m)$ for
a domain $\Omega \subset \R^m$ is said to be of a
\emph{vector field (or function) of bounded deformation}, 
if the distributional symmetrised gradient $Ew \in \Meas(\Omega; \Sym^2(\R^m))$
\cite{temam1985mpp}. We denote this space by $\BDspace(\Omega)$.
The concept can also be generalised to tensor fields of higher 
orders \cite{bredies2014symmetric}, useful for the definition
of $\TGV^k$ for $k > 2$.

Similar to $\BVspace$, we have the decomposition \cite{ambcosdal96}
\[
    Ew=\Eabs w \L^m + \Ejump w + \Ecantor w,
\]
where $\Eabs w$ is the absolutely continuous part.
For smooth functions
\[
    \Eabs w(x)=\frac{1}{2}\bigl(\grad w(x)+[\grad w(x)]^T\bigr).
\]
Generally this expression holds at points of \emph{approximate differentiability} 
of $w$, at $\L^m$-\ae $x \in \Omega$ \cite{ambcosdal96,hajlasz1996approximate}.
The jump part satisfies
\[
    \Ejump w = \frac{1}{2}\bigl(\nu_{J_u} \otimes (w^+ - w^-) + (w^+ - w^-) \otimes \nu_{J_w} \bigr) \H^{m-1} \restrict J_w,
\]
where the one-sided traces $w^\pm$, the jump set $J_w$ and its approximate
normal $\nu_{J_w}$ are as in the case of functions bounded variation.
Likewise, the Cantor part vanishes on any Borel set $\sigma$-finite 
with respect to $\H^{m-1}$.
Similarly to Proposition \ref{prop:bv-density}, defining
\[
    \tilde J_u \defeq \{x \in \Omega \mid \Theta_{*,m-1}(\abs{Eu}; x) >0\},
\]
we have
\begin{equation}
    \label{eq:bd-density}
    \tilde J_u \supset J_u, \quad\text{and}\quad
    \H^{m-1}(\tilde J_u \setminus J_u)=0.
\end{equation}

Many other results are however not as strong in $\BDspace(\Omega)$
as in $\BVspace(\Omega)$. For one, we only have \cite{ambcosdal96}
$\abs{Ew}(S_w \setminus J_w)=0$ instead of the stronger result
$\H^{m-1}(S_w \setminus J_w)=0$, which were to hold if
$u \in \BVspace(\Omega; \R^m)$. In fact, this result can be made
a little stronger. Namely, $\abs{Ew}(S_w \setminus J_v)=0$
for $v, w \in \BDspace(\Omega)$. 

Instead of Poincaré's inequality in $\BVspace(\Omega; \R^n)$, which says that
on Lipschitz domains we can approximate  zero-mean $\norm{u-\bar u}$ 
for $\bar u = \dashint_\Omega u \d y$ by $C_\Omega \abs{Du}(\Omega)$,
in $\BDspace(\Omega)$ we have the Sobolev-Korn inequality. This says
that there exists a constant $C_\Omega>0$ and for each $w \in \BDspace(\Omega)$
an element $\bar w \in \ker E$ such that
\[
    \norm{w -\bar w}_{2,L^1(\Omega; \R^m)} \le C_\Omega \norm{Ew}_{F,\Meas(\Omega; \Sym^2(\R^m))}.
\]
The kernel of $E$ consists of affine maps $\bar w(x)=Ax + c$
for $A$ a skew-symmetric matrix. The Sobolev-Korn inequality can
also be extended to symmetric tensor fields of higher-order than the 
present $k=1$, in which case the kernel is also a higher-order 
polynomial \cite{bredies2014symmetric}.

We will not really need these latter properties. The point is that
$\BDspace(\Omega)$ has significantly weaker regularity than
$\BVspace(\Omega; \R^m)$. This will have implications
to our work. What we will use is Korn's inequality,
which holds for $1<p<\infty$ but \emph{notoriously} not for $p=1$.
The form most suitable for our purposes, easily obtainable
from the versions in \cite{ambcosdal96,conti2005new,ciarlet2010korn},
states the existence of a constant $C_{\Omega,q}>0$ such that
\begin{equation}
    \label{eq:korn-inequality}
    \int_\Omega \norm{\grad w(x)}_F^q \d x \le 
    C_{\Omega,q}
        \int_\Omega \norm{\Eabs w(x)}_F^q \d x,
\end{equation}
for bounded domains $\Omega$, 
and vector fields $w \in W_0^{1,q}(\Omega; \R^m)$.
Our reason for the zero boundary condition, as opposed
to $\Omega=\R^m$, a Sobolev-Korn type $\norm{\grad(w-\bar w)(x)}_F^q$ 
on the left, or extra $\norm{w}_{2,L^q(\Omega; \R^m)}$ on the right,
is that in our application, we do not want to \emph{directly} enforce 
$w \in L^q(\Omega; \R^m)$. This will typically however follow 
a posteriori from \eqref{eq:korn-inequality} and the
Gagliardo-Nirenberg-Sobolev inequality.

\section{Problem statement}
\label{sec:problem}

Before stating our main results rigorously, we introduce
our assumptions on regularisation functionals and fidelities.
The definition of an admissible regularisation functional,
and our assumptions on the fidelity $\phi$ are unchanged from
Part 1, but we replace the double-Lipschitz comparability by
a notion of partial double-Lipschitz comparability, and limit
the set of admissible Lipschitz transformations to one that
operates along a specific direction.

\subsection{Admissible regularisation functionals and fidelities}

We begin by stating our assumptions on $R$, which are formulated
in Definition \ref{def:admissible} and Definition \ref{def:lipschitz-trans}.

\begin{definition}
    \label{def:admissible}
    We call $R$ an \emph{admissible regularisation functional}
    on $\Uspace$, where the domain $\Omega \subset \R^m$,
    if it is convex, lower semi-continuous with respect to weak*
    convergence in $\BVUspace$, 
    and there exist $C, c > 0$ such that
    \begin{equation}
	\label{eq:r-tv-bound}
	\norm{Du}_{2,\EUspace} \le C\bigl(1+\norm{u}_{\Uspace}+R(u)\bigr),
	\quad (u \in \Uspace).
    \end{equation}
\end{definition}

The next two technical definitions will be required by Definition \ref{def:lipschitz-trans}.

\begin{definition}
    We denote by $\LipClass(\Omega)$ the set of one-to-one Lipschitz
    transformations $\gamma: \Omega \to \Omega$ with
    $\inv \gamma$ also Lipschitz and both satisfying the Lusin $N$-property.
    With $U \subset \Omega$ an open set, and $z \in \R^m$ a unit vector, we set
    \[
        \begin{aligned}
        \LipClass(\Omega, U) & \defeq \{ \gamma \in \LipClass(\Omega) \mid \gamma(x)=x \text{ for } x \not\in U\},
        \quad\text{and}\\
        \LipClass(\Omega, U, z) & \defeq \{ \gamma \in \LipClass(\Omega, U) \mid P_z^\perp \gamma(y)=P_z^\perp y \text{ for all } y \in \Omega\}.
        \end{aligned}
    \]
    With $\gammaone, \gammatwo \in \LipClass(\Omega)$, 
    we then define the \term{basic double-Lipschitz comparison constants}
    \[
        \bitrans{\gammaone}{\gammatwo} 
        \defeq
        \sup_{x \in \Omega, \norm{v}=1} \norm{\lipjac{\gammaone}(x) v} + \norm{\lipjac{\gammatwo}(x) v} - 2\norm{v}.
    \]
    and
    \[
        \bitransjac{\gammaone}{\gammatwo} 
        \defeq
        \sup_{x \in \Omega} \abs{\jacobianf{m}{\gammaone}{x} + \jacobianf{m}{\gammatwo}{x} - 2}.
    \]
    Here the norm is the operator norm, $I$ the identity mapping on $\R^m$, and
    \[
        \lipjac{\gamma}(x) \defeq \grad \inv \gamma(\gamma(x)) \jacobianf{m}{\gamma}{x}.
    \]
    We also define the \emph{distances-to-identity}
    \[
        \didtrans{\gamma} \defeq \sup_{x \in \Omega} \norm{\grad \inv \gamma(\gamma(x))-I},
        \quad\text{and}\quad
        \didtransjac{\gamma} \defeq \sup_{x \in \Omega} \abs{\jacobianf{m}{\gamma}{x}-1},
    \]
    as well as the \emph{normalised transformation distance}
    \begin{equation}
        \label{eq:ntdis}
        \ntdis{\gamma} \defeq \sup_{\substack{U: \gamma \in \LipClass(\Omega, U) \\ u \in \BVspace(\Omega)}} 
        \int_\Omega \frac{\norm{\pf\gamma u(y) - u(y)}}{\diam(U)\abs{Du}(U)} \d y.
    \end{equation}
    Finally we combine these all into the overall
    \term{double-Lipschitz comparison constant}
    \[
        \bitransfull{\gammaone}{\gammatwo}
        \defeq \bitrans{\gammaone}{\gammatwo}
            + \bitransjac{\gammaone}{\gammatwo}
            + \didtrans{\gammaone}^2
            + \didtrans{\gammatwo}^2
            + \didtransjac{\gammaone}^2
            + \didtransjac{\gammatwo}^2
            + \ntdis{\gammaone}^2
            + \ntdis{\gammatwo}^2.
    \]
\end{definition}

Observe that by Poincaré's inequality, if $\support(\gamma-\iota)$ 
has Lipschitz boundary, then $\ntdis{\gamma} < \infty$ for
$\L^m$-\ae $x \in \Omega$, small enough $r>0$,
and $\gamma \in \LipClass(\Omega, U)$ for $U \subset \B(0, r)$.

\begin{definition}
    Given $u, v \in L^1(\Omega)$, and $\gamma \in \LipClass(\Omega)$,
    we define the \emph{partial pushforward}
    \[
        \partpf{\gamma}{u}{v} \defeq \pf\gamma(u-v)+v.
    \]
\end{definition}

Finally, we may state rigorously our most central concept.

\begin{definition}
    \label{def:lipschitz-trans}
    Let $x_0 \in \Omega$ and $u \in \BVspace(\Omega)$. We say that $R$ is
     \term{partially double-Lipschitz comparable for $u$ at $x_0$},
    if there exists a constant $\RCa>0$
    and a function $v \in W^{1,1}(\Omega)$,
    $x_0 \not\in S_v$,
    satisfying the following: for every $\epsilon>0$, 
    for some $r_0>0$, if $U \subset \B(x_0, r)$, $0<r<r_0$ and
    $\gammaone,\gammatwo \in \LipClass(\Omega, U)$
    with $\bitransfull{\gammaone}{\gammatwo} < 1$, then
    \begin{equation}
        \label{eq:partial-comp}
        R(\partpf{\gammaone}{u}{v})
        + R(\partpf{\gammatwo}{u}{v})
        - 2 R(u) \le 
        \RCa \bitransfull{\gammaone}{\gammatwo} \abs{D(u-v)}(\closure U)
        + (\bitransfull{\gammaone}{\gammatwo}^{1/2} +r)\epsilon r^m.
    \end{equation}
    We also say that $R$ is \term{partially double-Lipschitz comparable 
    at $x_0$ for $u$ in the direction $z$} for some unit vector $z \in \R^m$,
    if \eqref{eq:partial-comp} holds with the change that
    $\gammaone,\gammatwo \in \LipClass(\Omega, U, z)$. 
\end{definition}

\begin{remark}
    Usually $\RCa$ will be a universal constant for $R$, but we do not
    need this in this work. The function $v$ will depend on both $u$
    and $x_0$. The bound $\bitransfull{\gammaone}{\gammatwo} < 1$
    is mostly about aesthetics. We could instead allow
    $\bitransfull{\gammaone}{\gammatwo} < \delta$ for arbitrary
    $\delta>0$; we however cannot allow $\delta$ to be determined
    by $\epsilon>0$ for the proof of our main result 
    Theorem \ref{theorem:jumpset-strict}. It can only depend
    on $u$ and $x_0$ similarly to $v$.
    The only purpose of the bound is to allow the use of the single
    constant $\bitransfull{\gammaone}{\gammatwo}$ in front of both of the
    terms on the right hand side of \eqref{eq:partial-comp}, replacing
    any second-order terms that we might get in front of the remainder
    term $\epsilon r^m$ by first-order terms, which suffice there; 
    compare the proof of Proposition \ref{proposition:tgv2-admissibility}.
    For this the fixed bound suffices: 
    $\bitransfull{\gammaone}{\gammatwo} \le \bitransfull{\gammaone}{\gammatwo}^{1/2}$.
    Instead of this, we could also
    replace $\bitransfull{\gammaone}{\gammatwo}$ by two arbitrary polynomials
    of the square roots of the variables in its definition, the one 
    in front of $\abs{D(u-v)}(\closure U)$ being of lowest order $2$,
    and the one in front of $\epsilon r^m$ of lowest order $1$.
    Then we would not have to bound $\bitransfull{\gammaone}{\gammatwo}<1$.
    The reason for introducing the \emph{normalised} transformation
    distance is likewise aesthetical.
\end{remark}

We will strive to prove the following property of the 
regularisation functionals that we study. We will only
use the more involved case \ref{item:ass:comp-gamma}
in this work.

\begin{assumption}
    \label{ass:comp}
    We assume that $R$ is an admissible regularisation functional on $\Uspace$
    that satisfies the following for every $u \in \BVspace(\Omega)$
    and every Lipschitz $(m-1)$-graph $\Gamma \subset \Omega$.
    \begin{enumroman}
        \item\label{item:ass:comp-omega}
        $R$ is partially double-Lipschitz comparable for $u$ at $\L^m$-\ae $x \in \Omega$.
        \item\label{item:ass:comp-gamma}
        $R$ is partially double-Lipschitz comparable for $u$ in the direction
        $z_\Gamma$ at $\H^{m-1}$-\ae $x \in \Gamma$.
    \end{enumroman}
\end{assumption}

In order to show the existence of solutions to \eqref{eq:prob}, 
we require the following property from $\phi$.

\begin{definition}
    Let the domain $\Omega \subset \R^m$.
    We call $\phi: [0, \infty) \to [0, \infty]$ an \term{admissible fidelity
    function on $\Omega$} if it is convex, lower semi-continuous,
    $\phi(0)=0$, and satisfies for some $C > 0$ the coercivity
    condition
    \begin{equation}
        \label{eq:phi-l1-bound}
        \norm{u}_{\Uspace} \le C \left( \int_\Omega \phi(\abs{u(x)}) \d x + 1 \right),
        \quad
        (u \in \Uspace).
    \end{equation}
    Throughout this paper, we extend the domain of $\phi$ to $\R$
    by defining
    \[
        \phi(t) \defeq \phi(-t), \quad (t<0).
    \]
    This is in order to simplify the notation $\phi(\abs{u(x)})$
    to $\phi(u(x))$.
\end{definition}

For the study of the jump set $J_u$ of solutions to \eqref{eq:prob},
we require additionally the following increase criterion
to be satisfied by $\phi$.

\begin{definition}
    \label{def:phi-increase}
    We say that $\phi$ is \term{$p$-increasing} for $p \ge 1$,
    if there exists a constant $C_\phi>0$ for which
    \[
        \phi(x) - \phi(y) \le C_\phi (x-y)\abs{x}^{p-1},
        \quad
        (x, y \ge 0).
    \]
\end{definition}

As we have seen in Part 1, the functions $\phi(t)=t^p$, ($p \ge 1$),
in particular are $p$-increasing and admissible. Moreover, the 
problem \eqref{eq:prob} is well-posed under the above assumptions. 

\begin{theorem}[Part 1 \& standard]
    Let $f \in \Uspace$ satisfy $\int_\Omega \phi(f(x)) \d x < \infty$. 
    Suppose that $R$ is an admissible regularisation functional
    on $\Uspace$, and $\phi$ an admissible fidelity function for $\Omega$.
    Then there exists a solution $u \in \Uspace$ to \eqref{eq:prob},
    and any solution satisfies $u \in \BVUspace$.
\end{theorem}

\subsection{Jump set containment}

Our main result in this paper is the following theorem combined
with the corresponding partial double-Lipschitz comparability
estimates for higher-order regularisers in Section \ref{sec:reg}.

\begin{theorem}
    \label{theorem:jumpset-strict}
    Let the domain $\Omega \subset \R^m$ be bounded with Lipschitz boundary,
    and $\phi: [0, \infty) \to [0, \infty)$ be an admissible $p$-increasing
    fidelity function for some $1 < p < \infty$.
    Let $f \in \BVUspace \isect L^\infty_\loc(\Omega)$, and suppose
    $u \in \BVUspace \isect L^\infty_\loc(\Omega)$  solves \eqref{eq:prob}.
    If $R$ satisfies Assumption \ref{ass:comp}\ref{item:ass:comp-gamma}, then
    \[
        \H^{m-1}(J_u \setminus J_f)=0.
    \]
\end{theorem}

\begin{remark}
    Observe that we require $u$ to be locally bounded. This does not 
    necessarily  hold, and needs to be proved separately. In imaging 
    applications we are however not usually interested in unbounded data ,
    and nearly always $\norm{f}_{L^\infty(\Omega)} \le M$
    for some known dynamic range $M$.  So one would think that it
    suffices to add the constraint $\norm{u}_{L^\infty(\Omega)} \le M$
    to the problem \eqref{eq:prob}. This would even work under the
    simpler double-Lipschitz comparability \eqref{eq:double-lip}
    of Part 1, as the constraint is invariant under pushforwards
    $\pf\gamma u$. 
    
    It is, however, not generally invariant under the partial pushforward
    $\partpf{\gamma}{u}{v}$, which might not satisfy the constraint
    if $\abs{\tilde u(x_0)}=M$. If $\abs{\tilde u(x_0)}<M$, and the
    radius $r_0>0$ is small enough, the constraint will still be satisfied
    for otherwise well-behaved $u$ and typical constructions of $v$.
    What this says is that if (well-behaved) $u$ jumps outside $J_f$,
    then it will jump to activate the constraint. Whether in practise 
    the $v$ prescribed by the partial double-Lipschitz property of any 
    particular regulariser satisfies
    $\norm{\partpf{\gamma}{u}{v}}_{L^\infty(\Omega)} \le M$,
    is as interesting an open question as boundedness itself.
\end{remark}

The proof of Theorem \ref{theorem:jumpset-strict} is based on combining
the double-Lipschitz estimate for the regulariser with a separate
estimate for the fidelity, for specific ``shift'' transformations
$\gammax{\rho}{h}{r}$. In Part 1, we proved the following lemmas
about these.

\newcommand{\newu}[1][\rho]{\bar u_{#1,r}}
\newcommand{\newushift}[1][\rho]{\hat u_{#1,r}}
\newcommand{\newux}[2][\rho]{\bar {#2}_{#1,r}}
\newcommand{\gammarhoh}[1][\rho]{\gammax{#1}{h}{r}}
\newcommand{\gammarhohNEG}{\gammarhoh[-\rho]}

\begin{lemma}[Part 1]
    \label{lemma:f-transform-estim}
    Suppose $\gamma \in \LipClass(\Omega, U, z)$ for some
    $z \in \R^m$ and $U \subset \R^m$. Let $u \in \BVspace(\Omega)$. Then
    \[
        \int_U \abs{u(\gamma(x))-u(x)} \d x 
        \le M_{\gamma} \abs{Du}(U),
    \]
    where
    \[
        M_{\gamma} \defeq \sup_{x \in \Omega} \norm{\gamma(x)-x}.
    \]
\end{lemma}
\begin{proof}
    This is proved in Part 1 for specific transformations, but
    everything in the proof only depends on 
    $\gamma \in \LipClass(\Omega, U, z)$.
\end{proof}

\begin{lemma}[Part 1]
    \label{lemma:gamma-rho-h-r}
    Let $\Omega \subset \R^m$, and $\Gamma \subset \Omega$ 
    be a Lipschitz $(m-1)$-graph, $x_0 \in \Gamma$.
    There exist $r_0>0$ and Lipschitz transformations
    $\gammax{\rho}{h}{r} \in \LipClass(\Omega, U, z_\Gamma)$,
    ($-1 < \rho < 1$, $0 < r < r_0$), with
    \[
        U_r \defeq x_0 + z_\Gamma^\perp \isect \B(0, r) + (3 + \lip f_\Gamma) (-r, r) z_\Gamma.
    \]
    Moreover, there exists a constant $C>0$ such that 
    \[
        \bitransfull{\gammax{\rho}{h}{r}}{\gammax{-\rho}{h}{r}} \le C \rho^2.
    \]
\end{lemma}
\begin{proof}
    Only the facts $\didtransjac{\gammax{\rho}{h}{r}} \le C\abs{\rho}$
    and $\ntdis{\gammax{\rho}{h}{r}} \le C\abs{\rho}$, which
    are required for the bound $\bitransfull{\gammax{\rho}{h}{r}}{\gammax{-\rho}{h}{r}}$,
    are not directly proved in Part 1.
    The former follows immediately from the expression calculated for
    $\jacobianF{m}{\gammax{\rho}{h}{r}}$ in Part 1.
    Regarding $\ntdis{\gammax{\rho}{h}{r}}$, it follows from
    Lemma \ref{lemma:f-transform-estim} that
    \[
        \ntdis{\gamma} \le \sup_{U': \gamma \in \LipClass(\Omega, U', z)} \frac{M_\gamma}{\diam(U')},
        \quad (\gamma \in \LipClass(\Omega, U, z)).
    \]
    This is why we call $\ntdis{\gamma}$ the \emph{normalised} transformation distance.
    In Part 1, we proved that $M_{\gammax{\rho}{h}{r}} = \abs{\rho} r$.
    Therefore, there exists a constant $C>0$ satisfying
    \[
        \ntdis{\gammax{\rho}{h}{r}} \le \frac{\abs{\rho}r}{\diam(U_r)} \le C\abs{\rho}.
        \qedhere
    \]
\end{proof}

\begin{lemma}[Part 1]
    \label{lemma:phi-approx}
    Suppose $\phi$ is admissible and $p$-increasing with $1 < p < \infty$, and
    \emph{both} $u, f \in \BVspace(\Omega) \isect L^\infty_\loc(\Omega)$.
    Let $x_0 \in J_u \setminus (S_f \union Z_u)$.
    Then there exist $\theta \in (0, 1)$, $r_0>0$, independent of $\rho$,
    and a constant $C = C(\phi, u^\pm(x_0), \tilde f(x_0)) > 0$,
    such that whenever $0<r<r_0$ and $0 < \rho < 1$, the functions
    \begin{equation}
        \label{eq:p1-newu}
        \newu(x) = \theta u(x) + (1-\theta) \pf\gammarhoh u(x),
    \end{equation}
    satisfy
    \begin{equation}
        \label{eq:phi-approx}
        \int_{\Omega} \phi(\newu[\rho](x) -f(x)) \d x 
        +
        \int_{\Omega} \phi(\newu[-\rho](x) -f(x)) \d x 
        - 2\int_{\Omega} \phi(u(x) -f(x)) \d x
        \le - C \rho r^m.
    \end{equation}
\end{lemma}

With these, we may without much difficulty prove Theorem \ref{theorem:jumpset-strict}.

\begin{proof}[Proof of Theorem \ref{theorem:jumpset-strict}]
    Since $J_u$ is rectifiable, there exists a family $\{\Gamma_i\}_{i=1}^\infty$
    of Lipschitz graphs with $\H^{m-1}(J_u \setminus \Union_{i=1}^\infty \Gamma_i)=0$.
    If the conclusion of the theorem does not hold, that is
    $\H^{m-1}(J_u \setminus J_f)>0$, then for some $i \in \Z^+$
    and $\Gamma \defeq \Gamma_i$, also
    $\H^{m-1}((\Gamma \isect J_u) \setminus J_f)>0$.
    We will show that this leads to a contradiction.
    Since $R$ satisfies Assumption \ref{ass:comp}\ref{item:ass:comp-gamma},
    it is partially double-Lipschitz comparable in the direction $z_\Gamma$ for $u$
    at almost every $x_0 \in (\Gamma \isect J_u) \setminus J_f$. In particular,
    since $\H^{m-1}(Z_u)=0$, we may choose a point
    $x_0 \in (\Gamma \isect J_u) \setminus (J_f \union Z_u)$,
    where $R$ is also partially double-Lipschitz comparable in the 
    direction $z_\Gamma$ for $u$. We let $v$ be the function given by
    Definition \ref{def:lipschitz-trans}, and pick arbitrary $\epsilon>0$,
    $\theta \in (0, 1)$.
    Then for some $r_1>0$, every $U \subset \B(x_0, r)$, $0<r<r_1$ and
    $\gammaone,\gammatwo \in \LipClass(\Omega, U, z_\Gamma)$, 
    the estimate holds
    \begin{equation}
        \label{eq:mainthm-r-estim-1}
        R(\partpf{\gammaone}{u}{v}) + R(\partpf{\gammatwo}{u}{v}) - 2R(u)
        \le
        \RCa \bitransfull{\gammaone}{\gammatwo} \abs{D(u-v)}(\closure U)
        + (\bitransfull{\gammaone}{\gammatwo}^{1/2} +r) \epsilon r^m/(1-\theta).
    \end{equation}
    
    The overall idea in adapting the proof of the corresponding
    Theorem in Part 1 now is to apply Lemma \ref{lemma:phi-approx} 
    on the function $q \defeq u-v$ with data 
    $g \defeq f-v$ for $v$. Indeed
    \[
        \newushift \defeq \theta u + (1-\theta)\partpf{\gamma}{u}{v}
        =\big(\theta(u-v)+(1-\theta)\pf\gamma(u-v)\bigr)+v
        =\newux{q}+v,
    \]
    where $\newux{q}$ is defined by \eqref{eq:p1-newu}.
    It is important here that $v \in W^{1,1}(\Omega)$
    and $x_0 \not \in S_v$, so $J_g=J_f$ modulo a $\H^{m-1}$-null set
    and $(u-v)^+(x_0) - (u-v)^-(x_0)=u^+(x_0)-u^-(x_0)$.
    Thus by Lemma \ref{lemma:phi-approx}
    there exist $\theta \in (0, 1)$, $r_2>0$, and a constant 
    $C = C(\phi, u^\pm(x_0), \tilde f(x_0), v) > 0$,
    such that whenever $0<r<r_2$ and $0 < \rho < 1$, then
    \begin{equation}
        \notag
        \begin{split}
        \int_{\Omega} \phi(\newushift[\rho](x) -f(x)) \d x 
        +
        &
        \int_{\Omega} \phi(\newushift[-\rho](x) -f(x)) \d x 
        \\ &
        - 2\int_{\Omega} \phi(u(x) -f(x)) \d x
        \le - C \rho r^m.
        \end{split}
    \end{equation}
    By convexity, obviously
    \[
        R(\newushift[\rho]) + R(\newushift[-\rho]) - 2R(u)
        \le (1-\theta) 
        \left(
            R(\partpf{{\gammax{\rho}{h}{r}}}{u}{v})
            +
            R(\partpf{{\gammax{-\rho}{h}{r}}}{u}{v}) 
            -2R(u)
        \right),
    \]
    Since the transformations
    $\partpf{{\gammax{\rho}{h}{r}}}{u}{v} \in \LipClass(\Omega, U, z_\Gamma)$,
    and $U_r \subset \B(x, \kappa r)$ for some $\kappa > 0$,
    it follows from \eqref{eq:mainthm-r-estim-1}, for $0<r<r_1/\kappa$ that
    \begin{equation}
        \label{eq:mainthm-r-estim}
        R(\newushift[\rho]) + R(\newushift[-\rho]) - 2R(u)
        \le
        (1-\theta) \RCa \bitransfull{\gammax{\rho}{h}{r}}{\gammax{-\rho}{h}{r}} \abs{D(u-v)}(\closure U_r)
        + (\bitransfull{\gammax{\rho}{h}{r}}{\gammax{-\rho}{h}{r}}^{1/2} + \kappa r) \epsilon (\kappa r)^m.
    \end{equation}
    Since $x_0 \in J_u \setminus Z_u$, we have
    \[
        \abs{D(u-v)}(\closure U_r) \le 2 \abs{u^+(x_0)-u^-(x_0)} \omega_{m-1} (\kappa r)^{m-1}
    \]
    for $0 < r < r_3$ and some $r_3>0$. Lemma \ref{lemma:gamma-rho-h-r}
    gives
    \[
        \bitransfull{\gammax{-\rho}{h}{r}}{\gammax{-\rho}{h}{r}}
        \le C''' \rho^2
    \]
    for some constant $C''' >0$. Setting
    \[
        G(u) \defeq \int_{\Omega} \phi(u(x) -f(x)) \d x + R(u)
    \]
    and summing \eqref{eq:phi-approx} with \eqref{eq:mainthm-r-estim},
    we observe for some constants $C', C''>0$ 
    and every $0< r < \min\{r_1/\kappa, r_2, r_3\}$ and $0<\rho<1$ that
    \[
        G(\newushift[\rho](x))+G(\newushift[-\rho](x))-2G(u)
        \le
        C'\rho^2 r^{m-1} + C'' \rho \epsilon r^m - C \rho r^m + \epsilon (\kappa r)^{m+1}. 
    \]
    To see how to make the right hand side negative, let us set
    $\rho=\bar \rho r^m$. 
    Then we get
    \[
        G(\newushift[\rho](x))+G(\newushift[-\rho](x))-2G(u)
        \le
        (C'\bar\rho^2 + C'' \bar\rho \epsilon - C \bar \rho + \epsilon \kappa^{m+1}) r^{m+1}. 
    \]
    We first pick $\bar \rho$ small enough that
    $C'\bar\rho < C/4$. Then we pick $\epsilon>0$ small enough
    that $C''\epsilon<C/4$ and $\epsilon \kappa^{m+1} <\bar\rho C/4$. This
    will force $r>0$ small, but will give
    \[
        G(\newushift[\rho](x))+G(\newushift[-\rho](x))-2G(u)
        \le
        - C \bar \rho r^{m+1}/4,
    \]
    which is negative. This says that
    $\min\{G(\newushift[\rho](x)), G(\newushift[-\rho](x))\} < G(u)$.
    Thus we produce a contradiction to $u$ minimising $G$.
\end{proof}

\subsection{Approximation results}
\label{sec:approx}

In this section, we study study two aspects of approximation.
The first is how well we can approximate functions
of bounded deformation (or variation, for the matter) by
differentials $\grad v$ of functions $v \in W^{1,1}(\Omega)$.
These approximations form the basis of proving partial double-Lipschitz
comparability. The second aspect that we study is the
approximation of a function $u \in \BDspace(\Omega)$ in
terms of $\TGV$-strict convergence, or generally convergence
such that $u^i \to u$ weakly* in $\BVspace(\Omega)$ and
$\norm{Du^i-w}_{2,\Meas(\Omega; \R^m)} \to \norm{Du^i-w}_{2,\Meas(\Omega; \R^m)}$
for $w \in L^1(\Omega)$.

\subsection{Local approximation in $\BDspace(\Omega)$}
\label{sec:bdapprox}

One of our most critical concepts is stated in the following definition.

\begin{definition}
    We say that $w \in \BDspace(\Omega)$ is \emph{BV-differentiable}
    at $x \in \Omega$ if there exists $\bvdiff{w}{x} \in \BVspace_\loc(\Omega; \R^m)$
    such that
    \[
        \lim_{r \downto 0}
            \dashint_{\B(x, r)} \frac{\norm{w(y)-\bvdiff{w}{x}(y)}}{r} \d y =0.
    \]
\end{definition}

\begin{remark}
    Clearly $w$ is BV-differentiable if actually $w \in \BVspace_\loc(\Omega; \R^m)$.
    On a related note, the $\BVspace_\loc$ assumption was also required 
    in \cite{ambrosio-crippa-maniglia} for the study of traces of 
    another function $u$ with respect to $\abs{D^s w}$.
\end{remark}

\begin{proposition}
    Every $u \in \BDspace(\Omega)$ is BV-differentiable at $\L^m$-\ae
    $x \in \Omega$.
\end{proposition}

\begin{proof}
    We know from \cite{ambcosdal96,hajlasz1996approximate} that $w$ is 
    \emph{approximately differentiable} at $\L^m$-\ae $x \in \Omega$
    in the sense of existence of $L=\grad u(x) \in \R^{m \times m}$
    such that
    \[
        \lim_{r \downto 0}
            \dashint_{\B(x, r)} \frac{\norm{w(y)-w(x)-L(y-x)}}{r} \d y =0.
    \]
    It therefore suffices to set $\bvdiff{w}{x}(y) \defeq w(x)+L(y-x)$.
\end{proof}

\begin{remark}
    The domain of BV-differentiability is however potentially larger than
    approximate differentiability. A simple piece of evidence for this is 
    the fact that any $w \in \BVspace_\loc(\Omega; \R^m)$ is
    BV-differentiable everywhere, but not approximately differentiable
    on the jump set $J_w$.
    
    That we can show $\L^m$-\ae approximate differentiability is not
    entirely satisfying. We would prefer to have the property
    $\H^{m-1} \restrict J_w$-\ae Whether this can be achieved
    at least for $w$ a solution to \eqref{eq:tgv2-def},
    remains an interesting open question.
\end{remark}

We will need the following simple result for our main application
of BV-differentiability stated after it.

\begin{lemma}
    \label{lemma:bvdiff-jump}
    Suppose $w \in \BDspace(\Omega)$ is BV-differentiable at $x \in J_w$.
    Then $x \in J_{\bvdiff{w}{x}}$ with $w^+(x)=\bvdiff{w}{x}^+(x)$ and $w^-(x)=\bvdiff{w}{x}^-(x)$.
\end{lemma}
\begin{proof}
    By the definition of BV-differentiability
    \[
        \lim_{r \downto 0}
            \dashint_{\B(x, r)} \norm{w(y)-\bvdiff{w}{x}(y)} \d y
        \le
        \lim_{r \downto 0}
            \dashint_{\B(x, r)} \frac{\norm{w(y)-\bvdiff{w}{x}(y)}}{r} \d y = 0.
    \]
    This implies that $w$ and $\bvdiff{w}{x}$ have the same one-sided
    limits at $x$.
\end{proof}

The next lemma provides one of the most important ingredients
of our approach to proving partial double-Lipschitz comparability
for higher-order regularisation functionals.

\begin{lemma}
    \label{lemma:bd-w-v-approx}
    Let $w \in \BDspace(\Omega)$, $x \in \Omega$, and
    $\Gamma \ni x$ be a $C^1$ $(m-1)$-graph.
    Suppose that $w$ has traces $w^\pm(x)$ from both sides of $\Gamma$ at $x$,
    and $P_{z_\Gamma}^\perp(w^+(x)-w^-(x))=0$.
    Then there exists $v \in W^{1,1}_\loc(\Omega)$ with $x \not\in S_v$,
    satisfying
    \begin{equation}
        \label{eq:bd-w-v-approx-1}
        \lim_{r \downto 0} \dashint_{\B(x, r)} \norm{w-\grad v} \d y =0.
    \end{equation}
    If $w$ is moreover BV-differentiable at $x$, then given $\epsilon>0$,
    there exists $r_\epsilon>0$ such that
    every $U \subset \B(x, r)$, $0 < r < r_\epsilon$,
    and $\gamma \in \LipClass(\Omega, U, z_\Gamma)$
    satisfy
    \begin{equation}
        \label{eq:bd-w-v-approx-2}
        \int_{U} \norm{\pf\gamma(w-\grad v) - (w - \grad v)} \d y
        \le (\ntdis{\gamma} + r)\epsilon r^m.
    \end{equation}
    
    If $x \in \Omega \setminus S_w$, then we may take
    $\gamma \in \LipClass(\Omega, U)$ (without any specification of $\Gamma$).
\end{lemma}

\begin{proof}
    We first prove the results for $\gamma \in \LipClass(\Omega, U, z_\Gamma)$
    with $\Gamma$ specified.
    We denote for short $z \defeq z_\Gamma$, and let
    \[
        V^\pm \defeq \{ x \in \R^m \mid \pm \iprod{z}{x} > 0 \}.
    \]
    We define the transformation $\psi: \R^m \to \R^m$ by
    \[
        \psi(y) \defeq 
        y + f_\Gamma(P_z^\perp y) z,
    \]
    Then $\psi(V_\Gamma)=g_\Gamma(V_\Gamma)=\Gamma$.
    We observe also that
    \begin{equation}
        \label{eq:inv-psi}
        \inv\psi(y)=y - f_\Gamma(P_z^\perp y) z.
    \end{equation}
    Therefore
    \[
        \grad \inv \psi(y)  = I - (P_z^\perp)^* \grad f_\Gamma(P_z^\perp y) \otimes z,
    \]
    Since $\iprod{z}{(P_z^\perp)^* \grad f_\Gamma(P_z^\perp y)}=0$
    we find that $\grad \inv\psi(x)$ is invertible. Because
    $\grad f_\Gamma$ is by assumption continuous, this implies that
    \[
        \Psi(y) \defeq \grad \inv \psi(y)\grad \psi(\inv\psi(x))
        =
        \grad \inv \psi(y)\inv{[\grad \inv \psi(x)]}
    \]
    is continuous with $\Psi(x)=I$. More precisely
    for any $\epsilon>0$, for suitable $r_\epsilon>0$,
    \begin{equation}
        \label{eq:grad-inv-psi-approx}
        \norm{\Psi(y) -I } \le \epsilon,
        \quad (\norm{P_z^\perp(y-x)} \le r_\epsilon).
    \end{equation}
    
    We then let
    \[
        \begin{split}
        \bar v(y)
        &
        \defeq \iprod{[\grad \psi(\inv\psi(x))]^* w^+(x)}{y-\inv\psi(x)} \chi_{V^+}(y)
            \\
            &
            \phantom{\defeq}
             + \iprod{[\grad \psi(\inv\psi(x))]^* w^-(x)}{y-\inv\psi(x)} \chi_{V^-}(y),
        \end{split}
    \]
    and
    \[
        v \defeq \pf\psi \bar v.
    \]
    Recalling \eqref{eq:inv-psi}, we observe that $v$ is continuous and
    differentiable, $v \in W^{1,1}_\loc(\Omega) \isect C(\Omega)$
    and $x \not \in S_v$.
    This is the only place where we need the 
    assumption $P_{z_\Gamma}^\perp(w^+(x)-w^-(x))=0$.
    Defining
    \[
        w_0(y)
        \defeq w^+(x) \chi_{\Gamma^+}(y)
             + w^-(x) \chi_{\Gamma^-}(y),
    \]
    we get
    \[
        \grad v(y) = \Psi(y) w_0(y).
    \]
    Moreover, given $\epsilon>0$, by the definition of the one-sided
    limits $w^\pm(x)$, we have for some $r_\epsilon >0$ that
    \begin{equation}
        \label{eq:w-w0-epsilon}
        \dashint_{\B(x, r)} \norm{w(y)-w_0(y)} \d y \le \epsilon,
        \quad
        (0 < r < r_\epsilon).
    \end{equation}
    Thus with $C=2\max\{w^+(x), w^-(x)\}$,
    recalling \eqref{eq:grad-inv-psi-approx}, we obtain
    \begin{equation}
        \label{eq:w0-gradv-epsilon}
        \dashint_{\B(x, r)} \norm{w_0(y)-\grad v(y)} \d y
        \le
        \dashint_{\B(x, r)} \norm{\Psi(y)-I}\norm{w_0(y)} \d y
        \le
        C\epsilon,
        \quad
        (0 < r < r_\epsilon).
    \end{equation}
    Combined \eqref{eq:w-w0-epsilon} and \eqref{eq:w0-gradv-epsilon} give
    \[
        \dashint_{\B(x, r)} \norm{w(y)-\grad v(y)} \d y
        \le
        (1+C)\epsilon,
        \quad
        (0 < r < r_\epsilon).
    \]
    Since $\epsilon>0$ was arbitrary, we conclude 
    that \eqref{eq:bd-w-v-approx-1} holds.
    
    We now have to prove \eqref{eq:bd-w-v-approx-2},
    assuming that $w$ is BV-differentiable at $x$.
    We begin by observing that $q \defeq w-w_0$ is then BV-differentiable
    with $\bvdiff{q}{x}=\bvdiff{w}{x}-w_0$. Moreover
    \begin{equation}
        \label{eq:w0-bvdiff}
        \begin{split}
        \int_{U} \norm{\pf\gamma(\bvdiff{w}{x}-w_0)(y) - (\bvdiff{w}{x}(y)-w_0(y))} \d y
        &
        \le \ntdis{\gamma} \diam(U) \abs{D(\bvdiff{w}{x}-w_0)}(U).
        \\
        &
        \le C_m \ntdis{\gamma} r \abs{D(\bvdiff{w}{x}-w_0)}(U),
        \end{split}
    \end{equation}
    for some dimensional constant $C_m$ needed to apply 
    \eqref{eq:ntdis} to vector fields.
    By assumption $\gamma \in \LipClass(\Omega, U, z)$.
    Thus $P_z^\perp \inv\gamma(y)=P_z^\perp y$, which implies
    \begin{equation}
        \label{eq:psi-gamma-invar}
        \Psi \circ \inv \gamma = \Psi.
    \end{equation}
    Consequently
    \[
        \begin{split}
        \pf\gamma(\grad v - w_0) - (\grad v - w_0)
        &
        =
        \pf\gamma([\Psi-I]w_0) - [\Psi-I]w_0
        \\
        &
        =
        [\Psi-I](\pf\gamma w_0 - w_0).
        \end{split}
    \]
    Using \eqref{eq:grad-inv-psi-approx} again,
    \begin{equation}
        \label{eq:v-w0-shift-estim}
        \begin{split}
        \int_{U} \norm{\pf\gamma(\grad v - w_0)(y) & - (\grad v - w_0)(y)} \d y
        \\
        &
        \le
        \epsilon \int_{U} \norm{\pf\gamma w_0(y)-w_0(y)} \d y
        \\
        &
        \le
        \epsilon C_m \ntdis{\gamma} r \abs{Dw_0}(U)
        \le
        \ntdis{\gamma}  C' \epsilon r^m,
        \quad
        (0 < r < r_\epsilon),
        \end{split}
    \end{equation}
    for suitable $r_\epsilon>0$ and some constant $C'=C'(\Gamma, w^\pm(x_0))$.
    Choosing $r_\epsilon>0$ small enough, we may now for $0<r<r_\epsilon$
    finally approximate
    \begin{equation}
        \label{eq:w-v0-semifinal-estim}
        \begin{split}
        \int_{U} \norm{(\pf\gamma w(y) & - \pf\gamma\grad v(y)) - (w(y) - \grad v(y))} \d y
        \\
        &
        \le
        \int_{U} \norm{(\pf\gamma w_0(y)-\pf\gamma\grad v(y)) - (w_0(y) - \grad v(y))} \d y
            \\ & \phantom{\le}
            +
            \int_{U} \norm{(\pf\gamma\bvdiff{w}{x}(y)-\pf\gamma w_0(y)) - (\bvdiff{w}{x}(y)-w_0(y))} \d y
            \\ & \phantom{\le}
            + \int_{U} \norm{\pf\gamma\bvdiff{w}{x}(y)-\pf\gamma w(y)} \d y
            + \int_U \norm{\bvdiff{w}{x}(y)-w(y)} \d y
        \\
        &
        \le \ntdis{\gamma} r \abs{D(\bvdiff{w}{x}-w_0)}(U)
            + \ntdis{\gamma} C' \epsilon r^m
            + \epsilon r^{m+1}.
        \end{split}
    \end{equation}
    For the final inequality, we have used \eqref{eq:w0-bvdiff} and
    \eqref{eq:v-w0-shift-estim} for the two first terms on the left hand side,
    and the definition of BV-differentiability with the area formula
    for the last two terms.
    Referring to Lemma \ref{lemma:bvdiff-jump} and \eqref{eq:bd-density},
    we now observe for suitable $r_\epsilon>0$ that
    \[
        \abs{D(\bvdiff{w}{x}-w_0)}(U) \le \epsilon r^{m-1},
        \quad
        (0 < r < r_\epsilon).
    \]
    The arbitrariness of $\epsilon>0$ allows us to get rid
    of the constant factors in \eqref{eq:w-v0-semifinal-estim},
    and thus conclude the proof of \eqref{eq:bd-w-v-approx-2}
    in the case that $\gamma \in \LipClass(\Omega, U, z_\Gamma)$.

    If $x \in \Omega \setminus S_w$, and $\gamma \in \LipClass(\Omega, U)$ 
    (without any specification of $\Gamma$),
    we set
    \[
        \psi(y) \defeq y.
    \]
    Then
    \[
        \bar v(y) = v(y) = w_0(y) \equiv \tilde w(x).
    \]
    Also
    \[
        \Psi(y) = I,
    \]
    so we get
    \[
        \pf\gamma(\grad v - w_0) - (\grad v - w_0) = 0,
    \]
    and do not need the property \eqref{eq:psi-gamma-invar},
    which is the only place where we used the fact that
    $\gamma \in \LipClass(\Omega, U, z_\Gamma)$. Indeed,
    instead of \eqref{eq:grad-inv-psi-approx}, we have
    the stronger property
    \begin{equation}
        \notag
        \int_{U} \norm{\pf\gamma(\grad v - w_0)(y) - (\grad v - w_0)(y)} \d y = 0.
    \end{equation}
    The rest follows as before.
\end{proof}

\begin{remark}
    Our main reason for introducing the notion of BV-differentiability 
    is to be able to perform the pushforward approximation 
    in \eqref{eq:bd-w-v-approx-2}.
    For this it would also suffice to require the existence of
    $\bvdiff{w}{x,z}^\perp \in \BVspace_\loc(\Omega; z^\perp)$ satisfying
    \begin{equation}
        \label{eq:z-perp-bvdiff}
        \lim_{r \downto 0}
            \dashint_{\B(x, r)} \frac{\norm{P_z^\perp w(y)-\bvdiff{w}{x,z}^\perp(y)}}{r} \d y =0.
    \end{equation}
    Here $z \defeq z_\Gamma$.
    This holds because the slice $u_y^z(t) \defeq \iprod{z}{w(y+t z)} \in \BVspace(\Omega_y^z)$
    for $\Omega_y^z\defeq\{ t \in \R \mid y + t z \in \Omega\}$,
    and $\H^{m-1}$-\ae $y \in P_z^\perp \Omega$ \cite{ambcosdal96}.
    Unfortunately, we do not know much about the slices
    $t \mapsto \iprod{e}{w(y+t z)}$ for $e \perp z$, and therefore
    need to assume BV-differentiability.
    Combined with \eqref{eq:z-perp-bvdiff}, the assumption
    $P_{z_\Gamma}^\perp(w^+(x)-w^-(x))=0$ that we required
    reduces into $x \not \in S_{\bvdiff{w}{x,z}^\perp}$.
    In other words, we an approximate $P_z^\perp w$ by a BV function
    for which $x$ is a Lebesgue point.
    What all this means is that we need to assume extra regularity
    from $w$ in directions parallel to the plane $z_\Gamma^\perp$,
    but do not need to assume anything beyond
    $w \in \BDspace(\Omega)$ along $z_\Gamma$. 
    %
\end{remark}

\subsection{$\TGV$-strict smooth approximation}
\label{sec:tgv-smooth-approx}

We now study alternative forms of strict convergence in $\BVspace(\Omega)$.

\begin{theorem}
    \label{thm:tgv-smooth-approx}
    Suppose $\Omega \subset \R^m$ is open and let $(u, w) \in \BVspace(\Omega) \times \BDspace(\Omega)$. 
    Then there exists a sequence $\{(u^i, w^i)\}_{i=1}^\infty \in C^\infty(\Omega) \times C^\infty(\Omega; \R^m)$
    with
    \begin{align}
        \notag
        u^i & \to u \text{ in } L^1(\Omega), & \norm{Du^i-w^i}_{2,\Meas(\Omega; \R^m)} & \to \norm{Du-w}_{2,\Meas(\Omega; \R^m)}, \\
        \notag
        w^i & \to w \text{ in } L^1(\Omega; \R^m), & \norm{Ew^i}_{F,\Meas(\Omega; \Sym^2(\R^m))} & \to \norm{Ew}_{F,\Meas(\Omega; \Sym^2(\R^m))}.
    \end{align}
    If only $w \in L^1(\Omega; \R^m)$, then we get the three first converges,
    but not the fourth one.
\end{theorem}

\begin{proof}
    The proof follows the outlines of the approximation of $u \in \BVspace(\Omega)$ in
    terms of strict convergence in $\BVspace(\Omega)$, see \cite[Theorem 3.9]{ambrosio2000fbv}.
    We just have to add a few extra steps to deal with $w$, which is approximated similarly.
    
    To start with the proof, given a positive integer $m$, 
    we set $\Omega_0=\emptyset$ and
    \[
        \Omega_k \defeq \B(0, k+m) \isect \{x \in \Omega \mid \inf_{y \in \BD \Omega} \norm{x-y} \ge 1/(m+k)) \}.
    \]
    We pick $m$ large enough that
    \begin{equation}
        \label{eq:du-omega-one}
        \abs{D u-w}(\Omega \setminus \Omega_1) < 1/i
        \quad
        \text{and}
        \quad
        \abs{Ew}(\Omega \setminus \Omega_1) < 1/i.
    \end{equation}
    With
    \[
        V_k \defeq \Omega_{k+1} \setminus \closure \Omega_{k-1},
    \]
    each $x \in \Omega$ belongs to at most four sets $V_k$.
    We may then find a partition of unity $\{\zeta_k\}_{k=1}^\infty$ with
    $\zeta_k \in C_c^\infty(V_k)$, $0 \le \zeta_k \le 1$ and 
    $\sum_{k=1}^\infty \zeta_k \equiv 1$ on $\Omega$.
    
    With $\{\rho_\epsilon\}_{\epsilon > 0}$ a family of mollifiers, and $\epsilon_k > 0$,
    we let
    \[
        u_k \defeq \rho_{\epsilon_k} * (u \zeta_k),
        \quad
        \text{and}
        \quad
        w_k \defeq \rho_{\epsilon_k} * (w \zeta_k).
    \]
    We select $\epsilon_k > 0$ small enough that
    $\support u_k, \support w_k \subset V_k$ 
    (doable because $\zeta_k \in C_c^\infty(V_k)$), 
    and such that the estimates
    \begin{equation}
        \label{eq:smooth-approx-estim-u}
        \norm{u_k - u \zeta_k } \le 1/(2^ki),
        \quad
        \text{and}
        \quad
        \norm{\rho_{\epsilon_k} * (u \grad \zeta_k) - u \grad \zeta_k } \le 1/(2^k i),
    \end{equation}
    hold, as do
    \begin{equation}
        \label{eq:smooth-approx-estim-w}
        \norm{w_k - w \zeta_k } \le 1/(2^ki),
        \quad
        \text{and}
        \quad
        \norm{\rho_{\epsilon_k} * (w \otimes \grad \zeta_k) - w \otimes \grad \zeta_k } \le 1/(2^k i).
    \end{equation}
    We then let
    \[
        u^i(x) \defeq \sum_{k=1}^\infty u_k(x),
        \quad
        \text{and}
        \quad
        w^i(x) \defeq \sum_{k=1}^\infty w_k(x).
    \]
    By the construction of the partition of unity, for every $x \in \Omega$ there exists
    a neighbourhood of $x$ such the these sums have only finitely many non-zero terms.
    Hence $u^i \in C^\infty(\Omega)$.
    Moreover, as $u=\sum_{k=1}^\infty \zeta_k u$, \eqref{eq:smooth-approx-estim-u} gives
    \[
        \norm{u-u^i} \le \sum_{k=1}^\infty \norm{u_k - u \zeta_k} < 1/i.
    \]
    Thus $u^i \to u$ in $L^1(\Omega)$ as $i \to \infty$.
    Completely analogously $w^i \to w$ in $L^1(\Omega; \R^m)$ as $i \to \infty$.
    
    By lower semicontinuity of the total variation, we have
    \[
        \begin{aligned}
        \norm{Du-w}_{2,\Meas(\Omega; \R^m)} & \le \liminf_{i \to \infty} \norm{Du^i-w^i}_{2,\Meas(\Omega; \R^m)},
        \quad
        \text{and}
        \quad
        \\
        \norm{Ew}_{F,\Meas(\Omega; \Sym^2(\R^m))} & \le \liminf_{i \to \infty} \norm{Ew^i}_{F,\Meas(\Omega; \Sym^2(\R^m))}.
        \end{aligned}
    \]
    It therefore only remains to prove the opposite inequalities.
    Let $\varphi \in C_c^1(\Omega; \R^m)$ with $\sup_{x \in \Omega} \abs{\varphi(x)} \le 1$.
    We have
    \[
        \begin{split}
            \int_\Omega \divergence \varphi(x) u_k(x)
            &
            =
            \int_\Omega \divergence \varphi(x) (\rho_{\epsilon_k} * \zeta_k u)(x) \d x
            \\
            &
            =
            \int_\Omega \divergence (\rho_{\epsilon_k} * \varphi)(x) \zeta_k(x) u(x) \d x
            \\
            &
            =
            \int_\Omega \divergence [\zeta_k (\rho_{\epsilon_k} * \varphi)](x) u(x) \d x
            -
            \int_\Omega \iprod{\grad \zeta_k(x)}{(\rho_{\epsilon_k} * \varphi)(x)} u(x) \d x
            \\
            &
            =
            \int_\Omega \divergence [\zeta_k (\rho_{\epsilon_k} * \varphi)](x) u(x) \d x
            \\
            &
            \phantom{=}
            -
            \int_\Omega \iprod{\varphi(x)}{(\rho_{\epsilon_k} * (u\grad \zeta_k))(x)-(u\grad \zeta_k)(x)} \d x
            \\
            &
            \phantom{=}
            -
            \int_\Omega \iprod{\varphi(x)}{(u\grad \zeta_k)(x)} \d x.
        \end{split}
    \]
    Since $\sum_{k=1}^\infty \grad \zeta_k=0$, we have
    \[
        \sum_{k=1}^\infty \int_\Omega \iprod{\varphi(x)}{(u\grad \zeta_k)(x)} \d x = 0.
    \]
    Thus using \eqref{eq:smooth-approx-estim-u}, we get
    \[
        \begin{split}
            \int_\Omega \divergence \varphi(x) u^i(x)
            &
            =
            \sum_{k=1}^\infty \int_\Omega \divergence \varphi(x) u_k(x)
            \\
            &
            =
            \sum_{k=1}^\infty \int_\Omega \divergence [\zeta_k (\rho_{\epsilon_k} * \varphi)](x) u(x) \d x
            \\
            &
            \phantom{=}
            -
            \sum_{k=1}^\infty \left( \int_\Omega \iprod{\varphi(x)}{(\rho_{\epsilon_k} * (u\grad \zeta_k))(x)-(u\grad \zeta_k)(x)} \d x \right)
            \\
            &
            \le
            \sum_{k=1}^\infty \int_\Omega \divergence \varphi_k(x) u(x) \d x
            + 1/i.
        \end{split}
    \]
    In the final step, we have set
    \[
        \varphi_k \defeq \zeta_k (\rho_{\epsilon_k} * \varphi).
    \]
    By the definition of $w^i$, we also have
    \[
        \begin{split}
        \int_\Omega \iprod{\varphi(x)}{w^i(x)} \d x
        &
        =
        \sum_{k=1}^\infty \int_\Omega \iprod{\varphi(x)}{[\rho_{\epsilon_k} * (w \zeta_k)](x)} \d x
        \\
        &
        =
        \sum_{k=1}^\infty \int_\Omega \iprod{\varphi_k(x)}{w(x)} \d x.
        \end{split}
    \]
    Observing that $-1 \le \varphi_k \le 1$, and using the
    fact that $\sum_{k=1}^\infty \chi_{V_k} \le 4$, we further get
    \[
        \begin{split}
            \int_\Omega \varphi(x) \d[Du^i-w^i](x) 
            &
            =
            -\int_\Omega \divergence \varphi(x) u^i(x) + \iprod{\varphi(x)}{w^i(x)} \d x
            \\
            &
            \le
            -\int_\Omega \divergence \varphi_1(x) u(x) +\iprod{\varphi_1(x)}{w(x)} \d x
            \\
            &
            \phantom{\le}
            -
            \sum_{k=2}^\infty
            \int_\Omega \divergence \varphi_k(x) u(x) + \iprod{\varphi_k(x)}{w(x)} \d x
            + 1/i
            \\
            &
            \le
            \abs{Du-w}(\Omega)
            +
            \sum_{k=2}^\infty
            \abs{Du-w}(V_k)
            + 1/i
            \\
            &
            \le
            \abs{Du-w}(\Omega)
            +
            4 \abs{Du-w}(\Omega \setminus \Omega_1)
            + 1/i
            \\
            &
            \le
            \abs{Du-w}(\Omega)
            + 5/i.
        \end{split}
    \]
    In the final step we have used \eqref{eq:du-omega-one}.
    This shows that $\norm{Du^i-w^i}_{2,\Meas(\Omega; \R^m)} \to \norm{Du-w}_{2,\Meas(\Omega; \R^m)}$.
    
    Next we recall that
    \[
        \abs{Ew}(\Omega) = \sup_{\varphi \in C_c^\infty(\Omega; \Sym^{n \times n}) : \norm{\varphi(x)}_\infty \le 1} \int_\Omega \iprod{\divergence \varphi(x)}{w(x)} \d x
    \]
    with the divergence taken columnwise. Therefore, arguments analogous to the 
    ones above show that $\norm{Ew^i}_{F,\Meas(\Omega; \Sym^2(\R^m))} \to \norm{Ew}_{F,\Meas(\Omega; \Sym^2(\R^m))}$
    if $w \in \BDspace(\Omega)$. If only $w \in L^1(\Omega; \R^m)$, 
    then we do not get this converges, but the proof of the other
    converges did not depend on $w \in \BDspace(\Omega)$ at all.
    This concludes the proof.
\end{proof} 

For our present needs, the most important corollary of the above theorem is the following.

\begin{corollary}
    \label{corollary:w-smooth-approx}
    Suppose $\Omega \subset \R^m$ is open and let $(u, w) \in \BVspace(\Omega) \times L^1(\Omega; \R^m)$. 
    Then there exists a sequence $\{u^i\}_{i=1}^\infty \in C^\infty(\Omega)$ with
    \begin{equation}
        \label{eq:w-smooth-approx}
        u^i \to u \text{ in } L^1(\Omega) \quad \text{and} \quad \norm{Du^i-w}_{2,\Meas(\Omega; \R^m)} \to \norm{Du-w}_{2,\Meas(\Omega; \R^m)},
    \end{equation}
    as well as $Du^i \weaktostar Du$ weakly* in $\Meas(\Omega; \R^m)$.
\end{corollary}

\begin{proof}
    Let $\{(u^i, w^i)\}_{i=1}^\infty \in C^\infty(\Omega) \times C^\infty(\Omega; \R^m)$
    be given by Theorem \ref{thm:tgv-smooth-approx}. Then
    \[
        \begin{split}
        \lim_{i \to \infty}
        \norm{Du^i-w}_{2,\Meas(\Omega; \R^m)}
        &
        \le \lim_{i \to \infty} \bigl( \norm{Du^i-w^i}_{2,\Meas(\Omega; \R^m)}
            + \norm{w^i-w}_{2,L^1(\Omega; \R^m)} \bigr)
        \\
        &
        = \norm{Du-w}_{2,\Meas(\Omega; \R^m)}.
        \end{split}
    \]
    Analogously we deduce
    \[
        \lim_{i \to \infty}
        \norm{Du^i-w}_{2,\Meas(\Omega; \R^m)}
        \ge \norm{Du-w}_{2,\Meas(\Omega; \R^m)}.
    \]
    This gives \eqref{eq:w-smooth-approx}.
    Clearly, by moving to a subsequence of the original bounded sequence,
    we may further force $Du^i \weaktostar Du$ weakly* in $\Meas(\Omega; \R^m)$.
\end{proof}

The following corollary shows the approximability of $u \in \BVspace(\Omega)$ 
in terms of $\TGV^2$-strict convergence. It is of course easy to extend 
to $\TGV^k$ for $k>2$.

\begin{corollary}
    \label{cor:tgv-smooth-approx}
    Suppose $\Omega \subset \R^m$ is open and let $u \in \BVspace(\Omega)$. 
    Then there exists a sequence $\{u^i\}_{i=1}^\infty \in C^\infty(\Omega)$
    with $u^i \to u$ in $L^1(\Omega)$,
    $Du^i \weaktostar Du$ weakly* in $\Meas(\Omega; \R^m)$,
    and $\TGV^2_{(\beta,\alpha)}(u^i) \to \TGV^2_{(\beta,\alpha)}(u)$
    for any $\alpha, \beta > 0$.
\end{corollary}

\begin{proof}
    Let $w$ achieve the minimum in the differentiation cascade of
    definition \eqref{eq:tgv2-def} of $\TGV^2$, the minimiser
    existing by \cite{l1tgv}. Let then the sequence
    $\{(u^i, w^i)\}_{i=1}^\infty \in C^\infty(\Omega) \times C^\infty(\Omega; \R^m)$
    be given by Theorem \ref{thm:tgv-smooth-approx}.
    As in the proof of Corollary \ref{corollary:w-smooth-approx},
    we may assume that $Du^i \weaktostar Du$ weakly* in $\Meas(\Omega; \R^m)$.
    
    To see the convergence of $\TGV^2_{(\beta,\alpha)}(u^i)$ to
    $\TGV^2_{(\beta,\alpha)}(u)$, we observe that by definition
    \[
        \TGV^2_{(\beta,\alpha)}(u^i)
        \le \alpha \norm{Du^i-w^i}_{2,\Meas(\Omega; \R^m)} + \beta \norm{Ew^i}_{F,\Meas(\Omega; \Sym^2(\R^m))}
    \]
    Moreover
    \[
        \lim_{i \to \infty}\alpha \norm{Du^i-w^i}_{2,\Meas(\Omega; \R^m)} + \beta \norm{Ew^i}_{F,\Meas(\Omega; \Sym^2(\R^m))}
        = 
        \TGV^2_{(\beta,\alpha)}(u).
    \]
    Since the $\TGV^2$ functional is lower semicontinuous with respect to weak* convergence
    in $\BVspace(\Omega)$ (\cite{bredies2009tgv}, see also Lemma \ref{lemma:norm-equiv-tgv2}
    below), the claim follows.
\end{proof}

\section{Higher-order regularisers}
\label{sec:reg}

We now study partial double-Lipschitz comparability of second-
and higher-order regularisers. We start in Section \ref{sec:tgv2}
with $\TGV$, after which in Section \ref{sec:tgv2var} we consider
variants of $\TGV^2$ for which we have stronger results than
$\TGV^2$ itself. We finish in Section \ref{sec:ictv}
with infimal convolution $\TV$.

\subsection{Second-order total generalised variation}
\label{sec:tgv2}

Total generalised variation was introduced in \cite{bredies2009tgv}
as a higher-order extension of TV that avoids the stair-casing effect.
Following the \emph{differentiation cascade} formulation of
\cite{sampta2011tgv,l1tgv}, it may be defined for $u \in \BVUspace$ 
and $\alphavec=(\beta,\alpha)$ as
\begin{equation}
    \label{eq:tgv2-scalar-cascade}
    \TGV^2_{\alphavec}(u) \defeq
        \min_{w \in \Wspace}
        \alpha \norm{D u - w}_{2,\EUspace}
        +\beta \norm{E w}_{F,\EWspace},
\end{equation}
with a minimising $w \in \BDspace(\Omega)$ existing. 
Clearly
\[
    \TGV^2_{\alphavec}(u) \le \alpha \TV(u).
\]
Moreover, $\TGV^2_{\alphavec}$ is a seminorm.
In fact, it turns out that the norms $\norm{u}_{L^1}+\TV(u)$ 
and $\norm{u}_{L^1}+\TGV^2_\alphavec(u)$ are equivalent,
as shown in \cite{sampta2011tgv,l1tgv}. In other words
$\TGV^2_\alphavec$ induces the same topology in $\BDspace(\Omega)$
as $\TV$ does, but different geometry, as can be witnessed from
often much improved behaviour in practical image processing tasks.

\begin{lemma}
    \label{lemma:norm-equiv-tgv2}
    Let $\Omega \subset \R^m$ be a bounded domain with Lipschitz boundary.
    Then there exist constants $c,C > 0$, dependent on $\Omega$,
    such that for all $u \in L^1(\Omega)$ it holds
    \begin{equation}
        \label{eq:tgv2-equivalence}
        c \bigl( 
            \norm{u}_{\Uspace}
            + 
            \norm{Du}_{\EUspace}
        \bigr)
        \le
        \norm{u}_{F,1}+\TGV^2_{(\beta,\alpha)}(u)
        \le
        C \bigl( 
            \norm{u}_{\Uspace}
            + 
            \norm{Du}_{\EUspace}
        \bigr).
    \end{equation}
    Moreover, the functional $\TGV^2_{\alphavec}$ is lower semicontinuous
    with respect to weak* convergence in $\BVUspace$.
\end{lemma}

\begin{proof}
    Lower semicontinuity is proved in \cite{bredies2009tgv}
    for the original dual ball formulation. Equivalence to the 
    differentiation cascade formulation presented here is proved in
    \cite{sampta2011tgv,l1tgv}, where the norm equivalence is
    also proved.
\end{proof}

The following proposition states what we can say about partial double-Lipschitz
comparability of standard $\TGV^2$. Unfortunately, we cannot prove
Assumption \ref{ass:comp}\ref{item:ass:comp-gamma} quite exactly,
only for $\H^{m-1}$-\ae $x \in \Gamma \isect D_w \isect O_w^\Gamma$, where 
we denote by $D_w \subset \Omega$ the set of points where $w$ is 
BV-differentiable, and by $O_w^\Gamma$ the set of points $x \in \Gamma$
where $P_{z_\Gamma}^\perp(w^+(x)-w^-(x))=0$.

\begin{proposition}
    \label{proposition:tgv2-admissibility}
    Let $\Omega \subset \R^m$ be a bounded domain with Lipschitz boundary.
    Then $\TGV^2_\alpha$ is an admissible regularisation functional on $\Uspace$
    and satisfies Assumption \ref{ass:comp}\ref{item:ass:comp-omega}.
    Moreover, for any $u \in \BVspace(\Omega)$, a minimiser $w \in \BDspace(\Omega)$
    of \eqref{eq:tgv2-scalar-cascade}, and any Lipschitz 
    $(m-1)$-graph $\Gamma \subset \Omega$, the following holds.
    \begin{enumroman}
        \renewcommand{\theenumi}{(ii')}
        \renewcommand{\labelenumi}{(ii')}
        \item
        \label{item:tgv2:comp-gamma}
        $\TGV_\alphavec^2$ is partially double-Lipschitz comparable
        for $u$ in the direction $z_\Gamma$ at $\H^{m-1}$-\ae $x \in \Gamma \isect D_w \isect O_w^{\Gamma}$.
    \end{enumroman}
\end{proposition}

The basic idea of the proof is similar to the proof double-Lipschitz
comparability of $\TV$ in Part 1, but we need to deal with $w$ as well. 
This adds significant extra complications. One of them is the
use of the symmetrised gradient $Ew$, which does not allow us to use
estimates of the type in Lemma \ref{lemma:f-transform-estim}. 
We need the BV-differentiability of Section \ref{sec:bdapprox}
here. Also, the variable $w$ alone is problematic in the expression
$D\pf\gamma u - w$ for the use of the area formula. 
In order to deal with it, we have to take something, $\grad v$, 
out $w$, and shift this into $u$. Finally, we need to be careful 
with the jump set of $w$, also removing it from some estimates.

\begin{proof}
    We know from Lemma \ref{lemma:norm-equiv-tgv2} that  $\TGV^2_\alphavec$ 
    is lower semi-continuous with respect to weak* convergence in $\BVUspace$, 
    and that \eqref{eq:r-tv-bound} holds. 
    It therefore only remains to prove Assumption \ref{ass:comp}\ref{item:ass:comp-omega}
    and \ref{item:tgv2:comp-gamma}, that is partial double-Lipschitz comparability $\L^m$-\ae,
    and, for any given Lipschitz graph $\Gamma$, in the direction
    $z_\Gamma$ at $\H^{m-1}$-\ae $x \in \Gamma \isect D_w \isect O_w^\Gamma$.
    
    We pick arbitrary $w \in \BDspace(\Omega)$ achieving the minimum 
    in \eqref{eq:tgv2-scalar-cascade}.
    Regarding Assumption \ref{ass:comp}\ref{item:ass:comp-omega},
    we first of all observe that $\L^m(\Omega \setminus Q)=0$
    for $Q \defeq D_w \setminus S_w$. We claim that
    $\TGV^2_\alphavec$ is partially double-Lipschitz comparable
    for $u$ at every $x \in Q$. Regarding \ref{item:tgv2:comp-gamma},
    in order to apply Lemma \ref{lemma:bd-w-v-approx}, we need
    a $C^1$ $(m-1)$-graph.
    Indeed, as a consequence of the Whitney extension 
    theorem \cite[3.1.14]{federer1969gmt} and Lusin's theorem applied 
    to $f_\Gamma$, we may cover $\Gamma$ by $C^1$ $(m-1)$-graphs 
    $\{\Lambda_i\}_{i=1}^\infty$ satisfying $z_{\Lambda_i}=z_\Gamma$, 
    and $\H^{m-1}(\Gamma \setminus \Union_{i=1}^\infty \Lambda_i)=0$.
    If we show that $\TGV_\alphavec^2$ is partially double-Lipschitz
    comparable for $u$ in the direction $z_\Gamma$ at $\H^{m-1}$-\ae
    $x \in \Lambda_i \isect D_w \isect O_w^\Gamma$, for every $i \in \Z^+$, the claim will follow.
    
    To show Assumption \ref{ass:comp}\ref{item:ass:comp-omega},
    we apply Lemma \ref{lemma:bd-w-v-approx} at a point $x \in Q$. 
    To show \ref{item:tgv2:comp-gamma}, we apply the lemma at a point
    $x \in \Lambda_i \isect D_w \isect O_w^\Gamma$, ($i \in \Z^+$), 
    where the traces $w^\pm(x)$ from both sides of $\Lambda_i$ exist
    and $P_{z_\Gamma}^\perp(w^+(x)-w^-(x))=0$.
    This set, which we denote $Q_i$, satisfies 
    $\H^{m-1}((\Lambda_i \isect D_w \isect O_w^\Gamma) \setminus Q_i)=0$ 
    for each $i \in \Z^+$. This is exactly what we need. 
    
    We fix $x$ and let $U \subset \B(x, r)$ for suitable $r>0$.
    Lemma \ref{lemma:bd-w-v-approx} then gives us
    $v \in W^{1,1}(\Omega)$ with $x \not \in S_v$, and for each
    $\epsilon>0$ for $0 < r < r_\epsilon$  the estimates
    \begin{equation}
        \label{eq:tgv2-w-v-approx}
        \int_U \norm{\grad v-w} \d y \le \epsilon r^m,
    \end{equation}
    and
    \begin{equation}
        \label{eq:tgv2-w-v-gamma-approx}
        \int_{U} \norm{\pf\gamma(w-\grad v) - (w - \grad v)} \d y
        \le (\ntdis{\gamma} +r)\epsilon r^m/2.
    \end{equation}
    
    We define
    \[
        u_\gamma \defeq \partpf{\gamma}{u}{v} = \pf\gamma(u-v)+v,
        \quad
        (\gamma=\gammaone,\gammatwo).
    \]
    If we also set
    \begin{equation}
        \label{eq:def-g-tgv}
        G(u, w) \defeq \alpha \norm{Du-w}_{2,\Meas(\Omega; \R^m)} + \beta \norm{Ew}_{F,\Meas(\Omega; \Sym^2(\R^m)},
    \end{equation}
    then $\TGV^2_{\alphavec}(u') \le G(u', w')$ for all $(u', w') \in \BVspace(\Omega) \times \BDspace(\Omega)$.
    To prove partial double-Lipschitz comparability for $u$ at $x$, it therefore
    suffices to prove for any $\epsilon>0$ the existence of $r_0>0$
    such that for any $0 < r < r_0$, $U \subset \B(x, r)$, and
    $\gammaone,\gammatwo \in \LipClass(\Omega, U)$,
    resp.~ $\gammaone,\gammatwo \in \LipClass(\Omega, U, z_\Gamma)$,
    that
    \begin{equation}
        \label{eq:g-dlip}
        G(u_{\gammaone}, w)
        +
        G(u_{\gammatwo}, w)
        -
        2G(u, w)
        \le
        \bitransfull{\gammaone}{\gammatwo}\abs{(Du-v)}(\closure U)
        +
        (\bitransfull{\gammaone}{\gammatwo}^{1/2} + r)\epsilon r^m,
    \end{equation}
    for $w$ achieving the minimum in \eqref{eq:tgv2-scalar-cascade} for $u$.

    We suppose first that $u \in W^{1,1}(\Omega)$. 
    With $\gamma=\gammaone,\gammatwo$, by a lemma in Part 1, 
    we have $\pf\gamma\grad u = \grad \inv\gamma \pf\gamma \grad u$.
    Thus we may expand
    \[
        \begin{split}
        \grad u_\gamma - w
        &
        = \grad \pf\gamma(u-v) + \grad v -w
        \\
        &
        = \grad\inv\gamma \pf\gamma(\grad u-\grad v) + \pf\gamma(\grad v -w)
            \\ & \phantom{=}
            - [\pf\gamma (\grad v -w) - (\grad v - w)]
        \\
        &
        = \grad\inv\gamma \pf\gamma(\grad u-w) + (I-\grad\inv\gamma)\pf\gamma(\grad v -w)
            \\ & \phantom{=}
            - [\pf\gamma (\grad v -w) - (\grad v - w)].
        \end{split}
    \]
    It follows
    \begin{equation}
        \label{eq:tgv-dlip-1}
        \begin{split}
        \int_U \norm{\grad u_\gamma - w} \d y
        &
        \le
        \int_U \norm{\grad\inv\gamma(\gamma) (\grad u-w)} \jacobianF{m}{\gamma} \d y
            \\ & \phantom{\le}
            + \int_U \norm{(I-\grad\inv\gamma(\gamma)) (\grad v -w)} \jacobianF{m}{\gamma} \d y
            \\ & \phantom{\le}
            + \int_U \norm{\pf\gamma (\grad v -w) - (\grad v - w)} \d y.
        \end{split}
    \end{equation}
    With $\gamma=\gammaone,\gammatwo$,
    using \eqref{eq:tgv2-w-v-approx} we get
    \begin{equation}
        \label{eq:tgv-dlip-2}
        \begin{split}
        \int_U \norm{(I-\grad\inv\gamma(\gamma)) (\grad v -w)} \jacobianF{m}{\gamma} \d y
        &
        \le \didtrans{\gamma} \int_U \norm{\grad v -w} \jacobianF{m}{\gamma} \d y
        \\
        &
        \le
        \didtrans{\gamma} (\didtransjac{\gamma} + 1) \int_U \norm{\grad v -w} \d y
        \\
        &
        \le
        \didtrans{\gamma} (\didtransjac{\gamma} + 1) \epsilon r^m.
        \end{split}
    \end{equation}
    Using \eqref{eq:tgv-dlip-2} and \eqref{eq:tgv2-w-v-gamma-approx}
    in \eqref{eq:tgv-dlip-1}, we see that
    \begin{equation}
        \label{eq:tgv-gamma-approx-1}
        \begin{split}
        \int_U \norm{\grad u_\gamma - w} \d y
        &
        \le
        \int_U \norm{\lipjac{\gamma}(\grad u-w)} \d y
        + \didtrans{\gamma} (\didtransjac{\gamma} + 1) \epsilon r^m
        \\ & \phantom{=}
        + (\ntdis{\gamma} + r)\epsilon r^m/2.
        \end{split}
    \end{equation}
    Also by \eqref{eq:tgv2-w-v-approx}
    \[
        \int_U \norm{\grad u - w} \d y
        \le
        \int_U \norm{\grad u - \grad v} \d y + \epsilon r^m.
    \]
    Summing \eqref{eq:tgv-gamma-approx-1} for
    $\gamma=\gammaone,\gammatwo$, and subtracting $2\int_U \norm{\grad u - w} \d y$,
    we thus obtain
    \begin{equation}
        \label{eq:tgv2-dlip-w11-final}
        \begin{split}
        \int_U \norm{\grad u_\gammaone - w} \d y
        &
        +
        \int_U \norm{\grad u_\gammatwo - w} \d y
        -
        2\int_U \norm{\grad u - w} \d y
        \\
        &
        \le
        \bitrans{\gammaone}{\gammatwo} \int_U \norm{\grad u - \grad v} \d y
        +
        (C_{\gammaone,\gammatwo} +r)\epsilon r^m,
        \end{split}
    \end{equation}
    where
    \[
        C_{\gammaone,\gammatwo} \defeq
        \bitrans{\gammaone}{\gammatwo} +
        \didtrans{\gammaone} (\didtransjac{\gammaone} + 1)
        +
        \didtrans{\gammatwo} (\didtransjac{\gammatwo} + 1)
        +
        \ntdis{\gammaone} + \ntdis{\gammatwo}.
    \]
    Under the assumption $\bitransfull{\gammaone}{\gammatwo} < 1$
    contained in Definition \ref{def:lipschitz-trans},
    this can be made less than a constant times
    $\bitransfull{\gammaone}{\gammatwo}^{1/2}$.
    Since $\epsilon>0$ was arbitrary, we can get rid of any
    extra constant factors, proving \eqref{eq:g-dlip} 
    if $u \in W^{1,1}(\Omega)$.
    
    For general $u \in \BVspace(\Omega)$ we use an analogous smoothing argument 
    as in Part 1 for $\TV$. Namely, we use Corollary \ref{corollary:w-smooth-approx} 
    to approximate $u$ by a sequence $\{u^i\}_{i=1}^\infty \in C^\infty(\Omega)$ with
    \begin{equation}
        \label{eq:tgv-dlip-approx}
        u^i \to u \text{ in } L^1(\Omega) \quad \text{and} \quad \norm{D(u^i-v)}_{2,\Meas(\Omega; \R^m)} \to \norm{D(u-v)}_{2,\Meas(\Omega; \R^m)},
    \end{equation}
    as well as $Du^i \weaktostar Du$.
    Observe that $\epsilon>0$ in \eqref{eq:tgv2-dlip-w11-final}
    does not depend on $u$ itself, and neither does $r_0>0$ 
    nor the sets $Q$ and $Q_i$, ($i \in \Z^+$).
    Therefore \eqref{eq:tgv2-dlip-w11-final} holds in a uniform sense 
    for the sequence $\{u^i\}_{i=1}^\infty$. In particular
    \begin{equation}
        \label{eq:tgv-dlip-approx-2}
        G(u_\gammaone^i, w)
        +
        G(u_\gammatwo^i, w)
        -2G(u^i, w)
        \le
        \bitrans{\gammaone}{\gammatwo} \abs{D(u^i-v)}(U)
        + c
        \quad
        (i \in \Z^+)
    \end{equation}
    for the small nuisance variable 
    $c \defeq (\bitransfull{\gammaone}{\gammatwo}^{1/2} + r)\epsilon r^m$, 
    independent of $i$.
    Since \eqref{eq:tgv-dlip-approx} bounds he right hand side, we deduce 
    \[
        \TGV^2_\alphavec(u_\gammaone^i)
        +
        \TGV^2_\alphavec(u_\gammatwo^i)
        \le
        G(u_\gammaone^i, w)
        +
        G(u_\gammatwo^i, w)
         < \infty.
    \]
    By the $\BVspace$-coercivity in \eqref{eq:tgv2-equivalence}, we may
    therefore extract a subsequence, unrelabelled, such that both 
    $\{u_\gammaone^i\}_{i=1}^\infty$ and
    $\{u_\gammatwo^i\}_{i=1}^\infty$
    are convergent weakly* to some $\uone \in \BVspace(\Omega)$
    and $\utwo \in \BVspace(\Omega)$, respectively.
    Moreover, by \eqref{eq:tgv-dlip-approx}, \eqref{eq:tgv-dlip-approx-2},
    and the lower semicontinuity of the Radon norm with respect to weak* convergence, 
    we find that
    \begin{equation}
        \notag
        G(\uone, w)
        +
        G(\utwo, w)
        - 2 G(u, w)
        \le \liminf_{i \to \infty} \bitrans{\gammaone}{\gammatwo} \abs{D(u^i-v)}(U) + c.
    \end{equation}
    Let us pick an open set $U' \supset U$ such that $\abs{Du}(\BD U')=0$.
    Then $\abs{D(u^i -v)}(U') \to \abs{D(u -v)}(U')$ because $D(u^i -v) \to D(u -v)$ 
    strictly in $\Meas(\Omega; \R^m)$; see \cite[Proposition 1.62]{ambrosio2000fbv}. 
    It follows
    \[
        G(\uone, w)
        +
        G(\utwo, w)
        - 2 G(u, w)
        \le
        \bitrans{\gammaone}{\gammatwo} \abs{D(u -v)}(U') + c.
    \]
    By taking the intersection over all admissible $U' \supset U$,
    we deduce
    \begin{equation}
        \label{eq:tgv-almost-double-lip}
        G(\uone, w)
        +
        G(\utwo, w)
        - 2 G(u, w)
        \le
        \bitrans{\gammaone}{\gammatwo} \abs{D(u -v)}(\closure U) + c.
    \end{equation}
    This is almost \eqref{eq:g-dlip}  just have to
    show that $\uone = \partpf{\gammaone}{u}{v}$ and $\utwo = \partpf{\gammatwo}{u}{v}$. 
    Indeed
    \begin{equation}
        \label{eq:tgv-final-smooth-estim}
        \begin{split}
        \int_\Omega \abs{\uone(x) - \partpf{\gammaone}{u}{v}} \d x
        &
        \le
        \int_\Omega \abs{\uone(x) - \partpf{\gammaone}{u^i}{v}} \d x
        +
        \int_\Omega \abs{\partpf{\gammaone}{u}{v} - \partpf{\gammaone}{u^i}{v}} \d x
        \\
        &
        \le
        \int_\Omega \abs{\uone(x) - \partpf{\gammaone}{u^i}{v}} \d x
        +
        C \int_\Omega \abs{u(x) - u^i(x)} \d x
        \end{split}
    \end{equation}
    for
    \[
        C \defeq \left(\sup_x \jacobianf{m}{\gammaone}{x} \right)
        \le (\lip \gammaone)^m < \infty.
    \]
    The integrals on the right hand side of \eqref{eq:tgv-final-smooth-estim} 
    moreover tend to zero by the strict convergence of $u^i$ to $u$ and the 
    weak* convergence of $\pf\gammaone u^i$ to $\uone$. It follows that
    $\uone = \pf\gammaone u$. Analogously we show that $\utwo = \pf\gammatwo u$.
    The bound \eqref{eq:g-dlip} is now immediate from \eqref{eq:tgv-almost-double-lip}.
\end{proof}

\begin{remark}
    Since $w$ is kept fixed throughout,
    the proof of Proposition \ref{proposition:tgv2-admissibility}
    trivially extends to the differentiation cascade formulation
    of $\TGV^k$, ($k \ge 3$), defined for the parameter 
    vector $\alphavec=(\alpha_1,\ldots,\alpha_k) > 0$ as
    \[
        \TGV^k_\alphavec(u)=\inf_{
              \substack{
                u_\ell \in L^1(\Omega; \Sym^\ell(\R^m)); \\
                \ell=1,\ldots,k-1;\, u_0=u,\, u_k=0
              }
            }
            \sum_{\ell=1}^k \alpha_{k-\ell} \norm{E u_{\ell-1} - u_\ell}.
    \]
    The extension of the proof of this formulation in \cite{sampta2011tgv} for $k=2$
    to $k>2$ may be found in \cite{bredies2013regularization}.
\end{remark}

\subsection{Variants of $\TGV^2$}
\label{sec:tgv2var}

As we have seen, we are unable to prove jump set containment
for $\TGV^2$ unless we assume that the minimising $w \in \BDspace(\Omega)$
in \eqref{eq:tgv2-scalar-cascade} actually satisfies $w \in \BVspace_\loc(\Omega)$
and $P_{z_\Gamma}^\perp(w^+(x)-w^-(x))=0$ for any Lipschitz graph $\Gamma$.
Of course, we also have to assume that $u \in L^\infty_\loc(\Omega)$.
Whether we can prove any of these properties, we leave as
a fascinating question for future studies. Here we consider a couple
of variants of $\TGV^2$ for which at least $w \in \BVspace_\loc(\Omega)$,
and even $P_{z_\Gamma}^\perp(w^+(x)-w^-(x))=0$, which we recall having
denoted by $x \in O_w^\Gamma$.

The first modification, already considered in \cite{bredies2009tgv},
is the non-symmetric variant, which may be defined as
\begin{equation}
    \label{eq:tgv2ns-scalar-cascade}
    \TGVns^2_{\alphavec}(u) \defeq
        \min_{w \in \BVWspace}
        \alpha \norm{D u - w}_{F,\EUspace}
        +\beta \norm{D w}_{F,\DWspace}.
\end{equation}
It is not difficult to extend Lemma \ref{lemma:norm-equiv-tgv2}
to this this functional, and then repeat the proof of Proposition
\ref{proposition:tgv2-admissibility} to obtain the following.

\begin{proposition}
    Let $\Omega \subset \R^m$ be a bounded domain with Lipschitz boundary.
    Then $\TGVns^2_\alphavec$ is an admissible regularisation functional
    on $\BVUspace$ satisfying Assumption \ref{ass:comp}\ref{item:ass:comp-omega}.
    Moreover, for any $u \in \BVspace(\Omega)$, a minimiser $w \in \BDspace(\Omega)$
    of \eqref{eq:tgv2ns-scalar-cascade}, and any Lipschitz 
    $(m-1)$-graph $\Gamma \subset \Omega$, the following holds.
    \begin{enumroman}
        \renewcommand{\theenumi}{(ii'')}
        \renewcommand{\labelenumi}{(ii'')}
        \item
        \label{item:tgv2ns:comp-gamma}
        $\TGVns_\alphavec^2$ is partially double-Lipschitz comparable
        for $u$ in the direction $z_\Gamma$ at $\H^{m-1}$-\ae $x \in O_w^{\Gamma}$.
    \end{enumroman}
\end{proposition}

In fact, since the proof keeps $w$ fixed, we can do 
a little bit more.

\begin{proposition}
    \label{prop:f-psi}
    Let $\Omega \subset \R^m$ be a bounded domain with Lipschitz boundary.
    Suppose $\Psi: \BVspace(\Omega) \to \R$ is convex and lower semicontinuous 
    with respect to weak* convergence in $\BVWspace$, and satisfies for 
    some constant $C>0$ the  inequality
    \begin{equation}
        \label{eq:psi-coercive}
        \norm{Dw}_{F,\DWspace} \le C (1+ \Psi(w)).
    \end{equation}
    For any $\alphavec=(\beta,\alpha)>0$, define
    \begin{equation}
        \label{eq:f-psi-def}
        F_\Psi(u) \defeq
            \inf_{w \in L^1(\Omega; \R^m)}
            \alpha \norm{D u - w}_{F,\EUspace}
            +\beta \Psi(w),
        \quad (u \in \BVspace(\Omega)).
    \end{equation}
    Then $F_\Psi$ is an admissible regularisation functional
    on $\Uspace$ and satisfies Assumption \ref{ass:comp}\ref{item:ass:comp-omega}
    and \ref{item:tgv2ns:comp-gamma}.
\end{proposition}

\begin{proof}
    Minding \eqref{eq:psi-coercive}, it is not difficult to see 
    that a minimising sequence $\{w^i\}_{i=1}^\infty$ for
    the expression of $F_\Psi(u)$ in \eqref{eq:f-psi-def}
    is bounded in $\BVspace(\Omega)$.
    The existence of a minimising $w \in \BVWspace$ for $F_\Psi(u)$
    therefore follows from the lower semicontinuity of $\Psi$.
    If now $u^i \to u$ weakly* in $\BVspace(\Omega)$,
    with corresponding minimisers $w^i$ to the expression
    of $F_\Psi(u^i)$ in \eqref{eq:f-psi-def},
    then we may again deduce that $\{w^i\}_{i=1}^\infty$
    is bounded in $\BVWspace$. Therefore, we may extract
    a subsequence, unrelabelled, such that also $\{w^i\}_{i=1}^\infty$
    converge weakly* to some $u \in \BVUspace$. But the functional
    \begin{equation}
        \label{eq:def-g-psi}
        G(u, w) \defeq
        \alpha \norm{D u - w}_{F,\EUspace}
        +\beta \Psi(w),
    \end{equation}
    is clearly lower semicontinuous with respect to weak*
    convergence of both variables. Since $F_\Psi(u) \le G(u, w)$,
    and $F_\Psi(u^i)=G(u^i, w^i)$, we deduce that $F_\Psi$ is weak*
    lower semicontinuous. 
    
    To see the coercivity property \eqref{eq:r-tv-bound}, we
    use the fact that
    \[
        \norm{D u}_{F,\EUspace} 
        \le
        C\bigl(
        \norm{D u - w}_{F,\EUspace} 
        +
        \norm{D w}_{F,\EWspace} 
        +\norm{u}_{\Uspace}
        \bigl).
    \]
    This follows from the Poincaré inequality and an argument 
    by contradiction; for details see \cite{sampta2011tgv,l1tgv}.
    Plugging in \eqref{eq:psi-coercive} immediately
    proves \eqref{eq:r-tv-bound}.
    
    Finally, $F_\Psi$ is clearly convex, so the above considerations
    show that it is admissible. To prove Assumption \ref{ass:comp}, we
    adapt the proof of Proposition \ref{proposition:tgv2-admissibility},
    replacing $G$ defined by \eqref{eq:def-g-tgv} by that in \eqref{eq:def-g-psi}.
    Now $w$ is BV-differentiable everywhere, $D_w=\Omega$, so this
    part of the complications with $\TGV^2$ does not arise.
\end{proof}

As we recall from Korn's inequality, functions with bounded symmetrised
gradient in $L^q$ for $q>1$ are much better behaved than for $q=1$. 
We now want to exploit this to define variants of $\TGV^2$ with
stronger double-Lipschitz comparability properties.

\begin{corollary}
    \label{corollary:psi-lq}
    Suppose $\Omega \subset \R^m$ is a bounded open set with Lipschitz boundary.
    For $1 < q < \infty$, let
    \[
        \Psi(w) \defeq 
        \begin{cases}
            \norm{\Eabs w}_{F,L^q(\Omega; \Sym^2(\R^m))}, &
            w \in W_0^{1,q}(\Omega; \R^m), \\
            \infty, & \text{otherwise}.
        \end{cases}
    \]
    Then $\TGV_{\alphavec,0}^{2,q} \defeq F_\Psi$ is an admissible regularisation functional
    on $\Uspace$, satisfying Assumption \ref{ass:comp}.
\end{corollary}

\begin{proof}
    The condition \eqref{eq:psi-coercive} is an immediate
    consequence of Korn's inequality \eqref{eq:korn-inequality}.
    For weak* lower semicontinuity, we have to establish
    that any BV-weak* limit point $w$ of a sequence $\{w^i\}_{i=1}^\infty
    \subset W_0^{1,q}(\Omega; \R^m)$ with
    \begin{equation}
        \label{eq:psi-lq-bnd}
        \sup_i \norm{w}_{2,L^1(\Omega; \R^m)} + \norm{\Eabs w}_{F,L^q(\Omega; \Sym^2(\R^m))}
        \le C < \infty,
    \end{equation}
    also satisfies $w \in W_0^{1,q}(\Omega; \R^m)$. The lower 
    semicontinuity of $\norm{\Eabs \cdot}_{2,L^q(\Omega; \Sym^2(\R^m))}$ 
    itself is standard. By the Gagliardo-Nirenberg-Sobolev inequality,
    Korn's inequality \eqref{eq:korn-inequality}, and 
    approximation in $C_c^\infty(\Omega; \R^m)$, we also discover
    \[
        \norm{w^i}_{2,L^q(\Omega; \R^m)}
        +
        \norm{\grad w^i}_{F,L^q(\Omega; \Tensor^2(\R^m))}
        \le C' \norm{\Eabs w^i }_{F,L^q(\Omega; \Sym^2(\R^m))}
        \le C'C.
    \]
    We may therefore assume $\{w^i\}_{i=1}^\infty$ convergent
    weakly in $W^{1,q}(\Omega; \R^m)$, necessarily to $w$.
    It follows that $w \in W^{1,q}(\Omega; \R^m)$.
    But $w^i \in W_0^{1,q}(\Omega; \R^m)$ and
    $W_0^{1,q}(\Omega; \R^m)$ is strongly closed
    within $W^{1,q}(\Omega; \R^m)$, hence weakly
    closed as a convex set. Therefore $w \in W_0^{1,q}(\Omega; \R^m)$.
    This establishes BV-weak* lower semicontinuity of $\Psi$.
    Finally, we employ Proposition \ref{prop:f-psi}, noting
    that $\H^{m-1}(\Gamma \setminus O_w^\Gamma)=0$
    because $J_w=\emptyset$ and the (equal) one-sided traces
    exist $\H^{m-1}$-\ae on $\Gamma$ (by the BV trace theorem
    or $\H^{m-1}(S_w \setminus J_w)=0$).
\end{proof}

In Figure \ref{fig:q-comparison}, we have a simple comparison
of the effect of the exponent $q$ with fidelity $\phi(t)=t^2/$.
For $q=1$, we have chosen the base parameters $\alpha=25$ and 
$\beta=250$ on the image domain $\Omega \defeq [1, 256]^2$. For other 
values of $q$, namely $q=1.5$ and $q=2$, we have scaled $\beta$ by the
factor $256^{2(q-1)/q}$. This is what the the Cauchy-Schwarz inequality
gives as the factor for the $q$-norm to dominate the $1$-norm on an 
image with $256^2$ pixels.
We also include the $\TV$ result for comparison. The PSNR for
variants of $\TGV^2$ with different $q$ values is always the
same, $29.2$, while $\TV$ has PSNR $28.0$. There is also visually
no discernible difference between the different $q$-values,
whereas $\TV$ clearly exhibits the staircasing effect in the 
background sky. It therefore seems reasonable to also employ
in practise this kind of variants of $\TGV^2$, for which we have 
stronger theoretical results now, only lacking a proof of the
local boundedness of $u$ to complete the proof of the property 
$\H^{m-1}(J_u \setminus J_f)$.

\setlength{\w}{0.33\textwidth}
\def\SQl{148}
\def\SQb{130}
\def\SQr{212}
\def\SQt{194}
\newlength{\scf}
\def\shadowshift{0pt}
\newcommand{\drawzoomarea}{
    \draw[line width=1.5,dashed,color=red] 
        (\SQl\scf, \SQb\scf) rectangle (\SQr\scf, \SQt\scf);
}

\begin{figure}[t]
    \def\igraph#1{}
    \newcommand{\inc}[3][]{
        \setlength{\scf}{0.0039\w} 
        \begin{tikzpicture}
        \pgftext[at=\pgfpoint{0}{0},left,bottom]{
            \includegraphics[width=\w]{{#2}.png}
        }
        \pgftext[at=\pgfpoint{0.55\w}{0.05\w},left,bottom]{%
            \includegraphics[width=0.4\w,bb=148 130 212 194,clip]{{#2#3}.png}
        }
        \draw[line width=1.5, color=red] (0.55\w, 0.05\w) rectangle (0.95\w, 0.45\w);
        #1%
        \end{tikzpicture}
    }
    \begin{subfigure}[b]{\w}%
    \centering%
    \inc[\drawzoomarea]{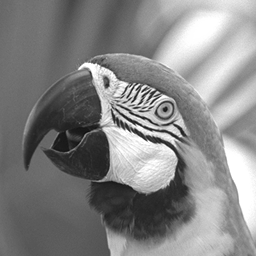}{}%
    \caption{Original}
    \end{subfigure}
    \begin{subfigure}[b]{\w}%
    \centering%
    \inc{tgvptest/output/parrot-noisy}{}%
    \caption{Noisy image}
    \end{subfigure}
    \begin{subfigure}[b]{\w}%
    \centering%
    \inc{tgvptest/output/parrot-tv1,25,0-l2}{}%
    \caption{$\TV$, $\alphavec=25$}%
    \end{subfigure}
    \begin{subfigure}[b]{\w}%
    \centering%
    \inc{tgvptest/output/parrot-tgv21,25,250-l2}{}%
    \caption{$q=1$, $\alphavec=(250,25)$}%
    \end{subfigure}
    \begin{subfigure}[b]{\w}%
    \centering%
    \inc{tgvptest/output/parrot-tgv2p1.5,25,10079-l2}{}%
    \caption{$q=1.5$, $\alphavec=(10079,25)$}%
    \end{subfigure}
    \begin{subfigure}[b]{\w}%
    \centering%
    \inc{tgvptest/output/parrot-tgv2p2.0,25,64000-l2}{}%
    \caption{$q=2.0$, $\alphavec=(64000,25)$}%
    \end{subfigure}
    \caption{Effect of the of exponent $q$ in the norm
        $\norm{\Eabs w}_{F,L^q(\Omega; \Sym^2(\R^m))}$
        in variants of $\TGV^2$ together with fidelity $\phi(t)=t^2/2$.
        The $\beta$ factor has been scaled from the case $q=1$ with 
        the help of the Cauchy-Schwarz inequality.
        There is no discernible difference between the
        results for different $q$, all having PSNR $29.2$,
        while the $\TV$ comparison has PSNR $28.0$
        and exhibits the stair-casing effect in the sky that
        $\TGV^2$ variants do not.
    }
    \label{fig:q-comparison}
\end{figure}
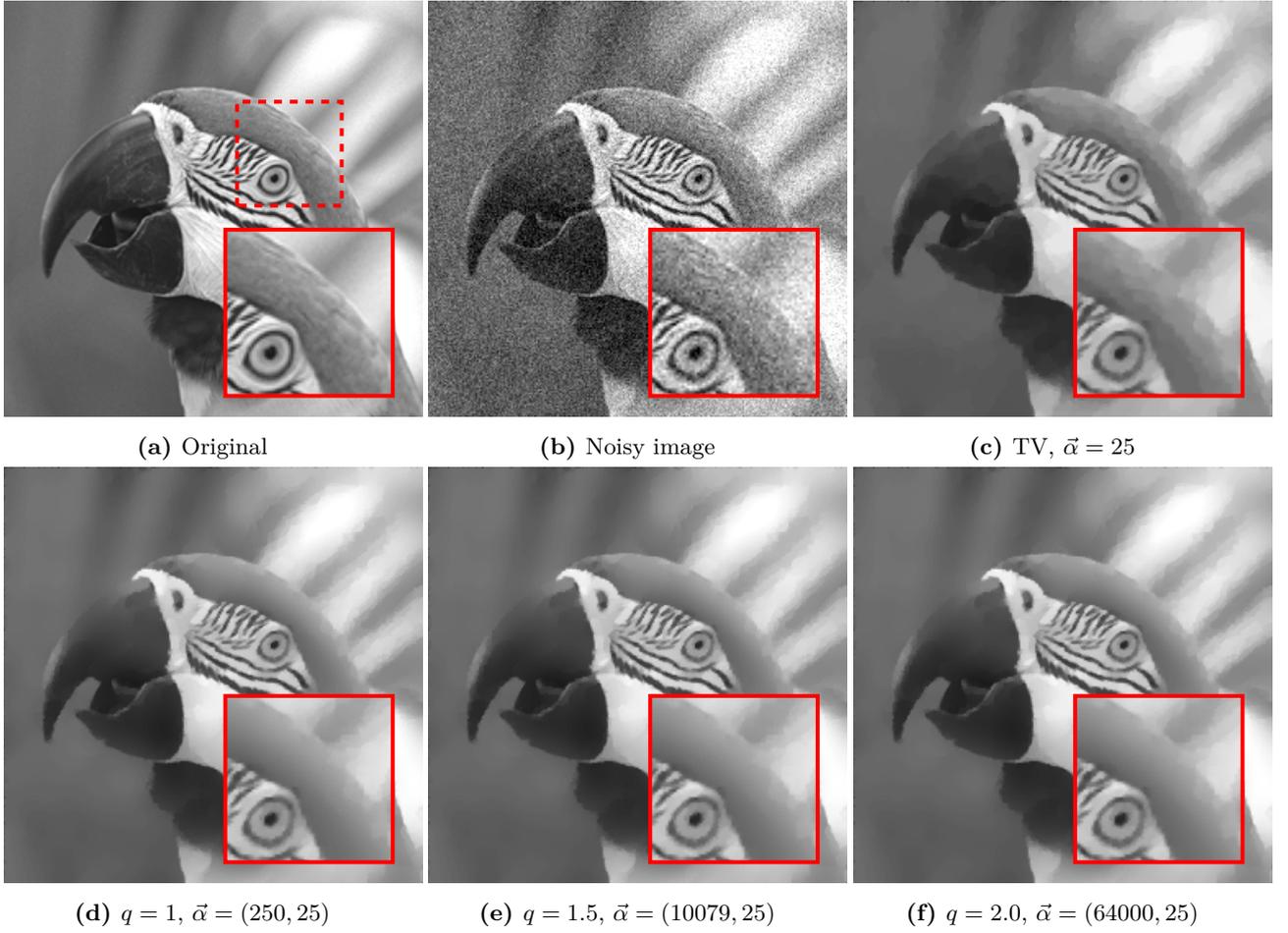

\subsection{Infimal convolution TV}
\label{sec:ictv}

Let $v \in W^{1,1}(\Omega)$ and $\grad v \in \BVWspace$. 
Define the second-order total variation by
\[
    \TV^2(w)=\norm{D \grad w}_{F,\DWspace}.
\]
Then second-order infimal convolution TV of $u \in \BVUspace$,
first introduced in \cite{chambolle97image}, is written 
\begin{equation}
    \label{eq:ictv}
    \ICTV_\alphavec(u) \defeq
    (\alpha\TV \IC \beta\TV^2)(u)
    \defeq
    \inf_{u=v^1+v^2}\bigl(\alpha \TV(v^1) + \beta \TV^2(v^2)\bigr),
\end{equation}
where necessarily $w \in W^{1,1}(\Omega)$,
$\grad w \in \BVWspace$, and $v \in \BVUspace$.
Clearly we have
\begin{equation}
    \label{eq:tgv2-ictv-tv}
    \TGV^2_{\alphavec}(u) \le \TGVns^2_{\alphavec}(u) \le \ICTV_\alphavec(u) \le \alpha \TV(u).
\end{equation}
It has been observed that while $\ICTV$ is better at
avoiding the stair-casing effect than $\TV$, it fares
worse than $\TGV^2$ \cite{benning2011higher}.

We did not find a proof of the weak* lower semi-continuity
of $\ICTV$ in the literature, so we provide one below.
Then we show that $\ICTV_\alphavec$ is admissible and
partially double-Lipschitz comparable as required by
Assumption \ref{ass:comp}.
As already observed in the Introduction, we remark, however, that 
the the jump set containment $\H^{m-1}(J_u \setminus J_f)=0$ can 
be proved for $\ICTV$ using the result for $\TV$.

\begin{lemma}
    \label{lemma:ictv-lsc}
    Let $\Omega \subset \R^m$ be a bounded domain with Lipschitz boundary.
    Then $\ICTV_\alphavec$ is lower semi-continuous with respect to weak* 
    convergence in $\BVUspace$.
\end{lemma}
\begin{proof}
    Let $u^i \weaktostar u$ weakly* in $\BVUspace$, ($i=0,1,2,\ldots$).
    We may then without loss of generality assume that
    $\{\norm{u^i}_{\Uspace}+\norm{Du^i}_{F,\EUspace}\}_{i=0}^\infty$
    is bounded. Let $v_1^i \in \BVUspace$ and $v_2^i \in W^{1,1}(\Omega)$ with
    $\grad w^i \in \BVWspace$ be such that
    \[
        \alpha \norm{Dv_1^i}_{F,\EUspace} + \beta \norm{D \grad v_2^i}_{F,\DWspace}
        \le
        \ICTV_\alphavec(u^i) + 1/i,
        \quad
        (i=0,1,2,\ldots).
    \]
    Observe that we may take each $v_1^i$ such that
    \begin{equation}
        \label{eq:ictv-vi-mean-zero}
        \bar v_1^i \defeq \int_\Omega v_1^i(x) \d x = 0,
    \end{equation}
    since the infimum in \eqref{eq:ictv} is independent of the mean of $v^1$
    and $v^2$.
    
    If $\limsup_i \ICTV_\alphavec(u^i) = \infty$, there is nothing
    to prove, so we may assume that $\sup_i \ICTV_\alphavec(u^i) < \infty$.
    It follows that both the sequence
    $\{\norm{D v_1^i}_{F,\EUspace}\}_{i=0}^\infty$ and the sequence
    $\{\norm{D \grad v_2^i}_{F,\EWspace}\}_{i=0}^\infty$ are bounded. 
    Minding \eqref{eq:ictv-vi-mean-zero} and the assumption
    that $\Omega$ has Lipschitz boundary, the Poincaré
    inequality now shows the existence of a constant $C>0$ such that
    \[
        \norm{v_1^i}_{\Uspace}
        =
        \norm{v_1^i-\bar v_1^i}_{\Uspace}
        \le
        C
        \norm{D v_1^i}_{F,\EUspace},
        \quad
        (i=0,1,2,\ldots),
    \]
    Consequently $\{v_1^i\}_{i=0}^\infty$ admits a subsequence, unrelabelled, 
    weakly* convergent in $\BVUspace$ to some $v \in \BVUspace$. 
    By the boundedness of the sequence
    $\{\norm{v_1^i}_{\Uspace}+\norm{Dv_1^i}_{F,\EUspace}\}\}_{i=0}^\infty$ and of
    $\{\norm{u^i}_{\Uspace}+\norm{Du^i}_{F,\EUspace}\}\}_{i=0}^\infty$
    it follows from $u^i=v_1^i+v_2^i$, moreover, that
    $\{\norm{v_2^i}_{\Uspace}+\norm{Dv_2^i}_{F,\EUspace}\}\}_{i=0}^\infty$ is bounded.
    Consequently, moving to a further subsequence, we may assume that
    $v_2^i \to v_2$ strongly in $W^{1,1}(\Omega)$ and
    $\grad v_2^i \weaktostar \grad v_2$ weakly* in $\EWspace$, for some $v_2 \in W^{1,1}(\Omega)$
    with $\grad v_2 \in \EWspace$.
    We clearly have $u=\lim u^i=\lim (v_1^i+v_2^i)=v_1+v_2$. Hence,
    by the lower semicontinuity of the Radon norm with respect
    to weak* convergence of measures, we obtain
    \[
        \begin{split}
        \ICTV_\alphavec(u) &
        \le
        \alpha \norm{Dv_1}_{F,\EUspace} + \beta \norm{D \grad v_2}_{F,\DWspace}
        \\
        &
        \le
        \liminf_{i \to \infty}
        \alpha \norm{Dv_1^i}_{F,\EUspace} + \beta \norm{D \grad v_2^i}_{F,\DWspace}
        \\
        &
        \le \liminf_{i \to \infty} \bigl(\ICTV_\alphavec(u^i) + 1/i\bigr)
        \\
        &
        = \liminf_{i \to \infty} \ICTV_\alphavec(u^i).
        \end{split}
    \]
    Thus $\ICTV_\alphavec$ is weak* lower semi-continuous, as claimed.
\end{proof}

\begin{proposition}
    Let $\Omega \subset \R^m$ be a bounded domain with Lipschitz boundary.
    Then $\ICTV_\alphavec$ is an admissible regularisation functional
    on $\Uspace$, and satisfies Assumption \ref{ass:comp}.
\end{proposition}
\begin{proof}
    We have already proved that $\TGV_\alphavec^2$ satisfies the
    coercivity criterion \eqref{eq:r-tv-bound}. It immediately
    follows from \eqref{eq:tgv2-ictv-tv} that $\ICTV_\alphavec$ 
    also satisfies this. By Lemma \ref{lemma:ictv-lsc} ICTV
    is weak* lower semicontinuous in $\BVspace(\Omega)$.
    The rest of the conditions of Definition \ref{def:admissible}
    are obvious. Thus $\ICTV_\alphavec$ is admissible.
    Assumption \ref{ass:comp} can be proved following the proof 
    of Proposition \ref{proposition:tgv2-admissibility} as follows.
    Instead of \eqref{eq:def-g-tgv}, we define
    \begin{equation}
        \notag
        G(u, w) \defeq \alpha \norm{Du-w}_{2,\Meas(\Omega; \R^m)} + \beta \norm{Dw}_{F,\Meas(\Omega; \Tensor^2(\R^m)}.
    \end{equation}
    With $u=v_1+v_2$ a minimising decomposition in \eqref{eq:ictv},
    we set $w=\grad v_2$ and $v=v_2$, observing that
    the conclusions of Lemma \eqref{lemma:bd-w-v-approx} hold trivially
    for $v=v_2$ at every Lebesgue point $x$ of $v_2$. 
    Since $v_2 \in W^{1,1}(\Omega)$, this is in particular the case
    for $\H^{m-1}$-\ae $x \in \Gamma$
    for any given Lipschitz $(m-1)$-graph $\Gamma$.
    Thus we have no complications as in the case of $\TGV^2$.
    It follows that Assumption \ref{ass:comp} holds.
\end{proof}


\section{Limiting behaviour of $L^p$-$\TGV^2$}
\label{sec:lptgv2-bounds}

Having studied qualitatively the behaviour of the jump set
$J_u$, and obtained good results for variants of $\TGV^2$
although not $\TGV^2$ itself, we now want to study it
quantitatively. We let the second regularisation parameter
$\beta$ of $\TGV^2$ go to zero, and see what happens to
$D^s u$ for $u$ solution to the $L^p$-$\TGV^2$ regularisation
problem, namely
\begin{equation}
    \label{eq:lp-tgv2}
    \min_{u \in \BVUspace} \norm{f-u}_{L^p(\Omega)}^p + \TGV_\alphavec^2(u).
\end{equation}
The next proposition states our findings.

\begin{proposition}
    \label{prop:jump-mass-bound}
    Let $\alpha > 0$, $1 \le p < \infty$, and $f \in L^p(\Omega)$.
    Then for every $\epsilon > 0$ there exists $\beta_0 > 0$ such
    that any solution $u$ to \eqref{eq:lp-tgv2} satisfies
    \begin{equation}
        \label{eq:jump-mass-bound}
        \norm{D^s u}_{F,\EUspace} < \epsilon
        \quad
        \text{ for }
        \beta \in (0, \beta_0).
    \end{equation}
\end{proposition}
\begin{proof}
    Let $\{\rho_\tau\}_{\tau > 0}$ be the standard
    family of mollifiers on $\R^m$, and use the 
    notation
    \[
        u_\tau \defeq \rho_\tau * u,
        \quad
        \text{ and }
        \quad
        w_\tau \defeq \rho_\tau * w
    \]
    for mollified functions, where $w$ minimises
    \eqref{eq:tgv2-scalar-cascade} for $u$. Then
    \[
        \begin{split}
        \norm{f-u_\tau}_{L^p(\Omega)}
        &
        \le
        \norm{f_\tau-u_\tau}_{L^p(\Omega)}
        +
        \norm{f-f_\tau}_{L^p(\Omega)}
        \\
        &
        \le
        \norm{f-u}_{L^p(\Omega)}
        +
        \norm{f-f_\tau}_{L^p(\Omega)}.
        \end{split}
    \]
    Let $\delta > 0$ be arbitrary.
    Since $\norm{f-f_\tau}_{L^p(\Omega)} \to 0$,
    we deduce the existence of $\tau_\delta > 0$ such
    that
    \[
        \norm{f-u_\tau}_{L^p(\Omega)}^p
        \le
        \norm{f-u}_{L^p(\Omega)}^p
        +
        \delta,
        \quad
        \text{ for }
        \tau \in (0, \tau_\delta].
    \]
    It can easily be shown by application of Green's identities
    that the symmetric differential operator $E$ satisfies
    \[
        E w_\tau=\rho_\tau * Ew = \Eabs \rho_\tau * w,
    \]
    similarly to corresponding well known results on the operator $D$.
    With
    \[
        \tilde w \defeq \grad u_\tau
    \]
    we thus obtain for some constant $C > 0$ the estimate
    \[
        \begin{split}
        \norm{E \tilde w}_{F,\DWspace}
        & 
        \le
        \norm{E \tilde w - E w_\tau}_{F,\DWspace}
        +
        \norm{E w_\tau}_{F,\DWspace}
        \\
        &
        \le
        \norm{\Eabs \rho_\tau * (D u - w)}_{F,\DWspace}
        +
        \norm{E w}_{F,\DWspace}
        \\
        &
        \le
        C \inv \tau
        \norm{D u - w}_{F,\EUspace}
        +
        \norm{E w}_{F,\DWspace}.
        \end{split}
    \]
    As
    \[
        \norm{D u_\tau - \tilde w}_{F,\EUspace}=0,
    \]
    it follows that
    \[
        \begin{split}
            \norm{f-u_\tau}_{L^p(\Omega)}^p + \TGV_\alphavec^2(u_\tau)
            &
            \le
            \norm{f-u_\tau}_{L^p(\Omega)}^p
            +
            \alpha \norm{D u_\tau - \tilde w}_{F,\EUspace}
            +
            \beta \norm{E \tilde w} 
            \\
            &
            \le
            \norm{f-u}_{L^p(\Omega)}^p
            +
            \delta
            +
            C \beta \inv \tau
            \norm{D u - w}_{F,\DWspace}
            \\
            &
            \phantom{\le}
            +
            \beta \norm{E w}_{F,\DWspace}
            \\
            &
            \le
            \norm{f-u}_{L^p(\Omega)}^p + \TGV_\alphavec^2(u)
            \\
            &
            \phantom{\le}
            +
            (C \beta \inv \tau -\alpha)
            \norm{D u - w}_{F,\EUspace}
            +
            \delta.
        \end{split}
    \]
    Consequently $u_\tau$ provides a contradiction to $u$
    being a solution to \eqref{eq:lp-tgv2} if
    \[
        \delta
        <
        (\alpha - C \beta \inv \tau)
        \norm{D u - w}_{F,\EUspace}.
    \]
    Since
    \[
        \norm{D^s u}_{F,\EUspace} \le \norm{D u - w}_{F,\EUspace},
    \]
    it follows that for an optimal solution $u$, it must hold
    \[
        \norm{D^s u}_{F,\EUspace} 
        \le \delta/(\alpha - C \beta \inv \tau_\delta)
    \]
    Thus \eqref{eq:jump-mass-bound} holds if
    \[
        \delta + C \beta \inv \tau_\delta \epsilon < \epsilon \alpha.
    \]
    Choosing $\delta < \epsilon \alpha$, we find
    $\beta_0 > 0$ such that this is satisfied for $\beta \in (0, \beta_0)$.
\end{proof}

We illustrate numerically in Figure \ref{fig:square-example-l1-10} to 
Figure \ref{fig:square-example-l2-10} the implications of 
Proposition \ref{prop:jump-mass-bound} and Theorem \ref{theorem:jumpset-strict} 
on a very simple test image with a square in the middle.
We did the experiments for fixed $\alpha=10$ or $\alpha=5$ 
and varying $\beta$, with fidelity $\phi(t)=t^p$ for $p=1$ and $p=2$.
In all cases, as $\beta$ goes down from a large value with
good reconstruction, the image first starts to smooth out.
This happens until the smallest $\beta$, for which we 
appear to have recovered $f$! This may seem a little counterintuitive,
as Proposition \ref{prop:jump-mass-bound} forbids big jumps 
for small $\beta$. But we should indeed have very steep gradients near the
boundary. These are lost in the discretisation.

Besides this, the numerical experiments verify $\H^{m-1}(J_u \setminus J_f)$
for $p=2$, and demonstrate the fact that it does not hold for $p=1$.
However, the set $J_u \setminus J_f$ has specific curvature dependent
on the parameter $\alpha$. For $p=2$, we of course observe the
well-known phenomenon of contrast loss. In the corners
where $p=1$ starts to produce new jumps, $p=2$ starts to smooth
out the solution, also not reconstructing the jumps of the corners.

\setlength{\w}{0.18\textwidth}

\begin{figure}
    \centering
    \input{test/square-example-l1-10.tex}
    \caption{Illustration of varying $\beta$ parameter for $\TGV^2$
        regularisation with $L^1$ fidelity on $f=\chi_{(-32, 32)^2}$. 
        For $\beta=0.1$ in \subref{fig:square-example-l1-10-1} it appears that we have
        recovered  $f$. The apparent full recovery in \subref{fig:square-example-l1-10-1} for $\beta=0.1$
        is an effect  of the discretisation. For $\beta=50$ in \subref{fig:square-example-l1-10-8} 
        due to numerical difficulties we have not fully recovered the corners 
        (of curvature $1/\alpha=0.2$) that should start to become sharp.
    }
    \label{fig:square-example-l1-10}
\end{figure}

\begin{figure}
    \centering
    \input{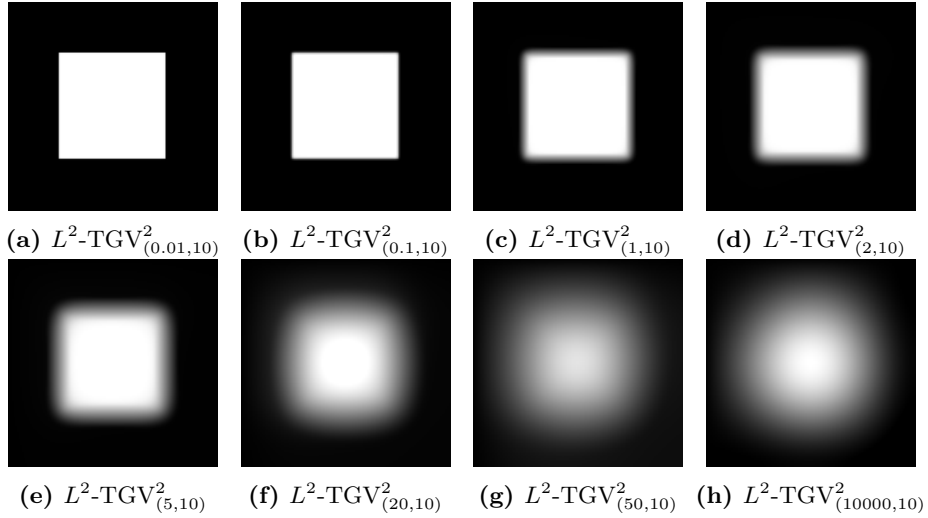}
    \caption{Illustration for $\alpha=10$ of varying $\beta$ parameter for $\TGV^2$
        regularisation with squared $L^2$ fidelity on $f=\chi_{(-32, 32)^2}$,
        to compare with Figure \ref{fig:square-example-l1-10} for $L^1$ fidelity.
    }
    \label{fig:square-example-l2-10}
\end{figure}

\begin{figure}
    \centering
    \input{test/square-example-l1-5.tex}
    \caption{Illustration for $\alpha=5$ of varying $\beta$ parameter for $\TGV^2$
        regularisation with $L^1$ fidelity on $f=\chi_{(-32, 32)^2}$.
    }
    \label{fig:square-example-l1-5}
\end{figure}

\begin{figure}
    \centering
    \input{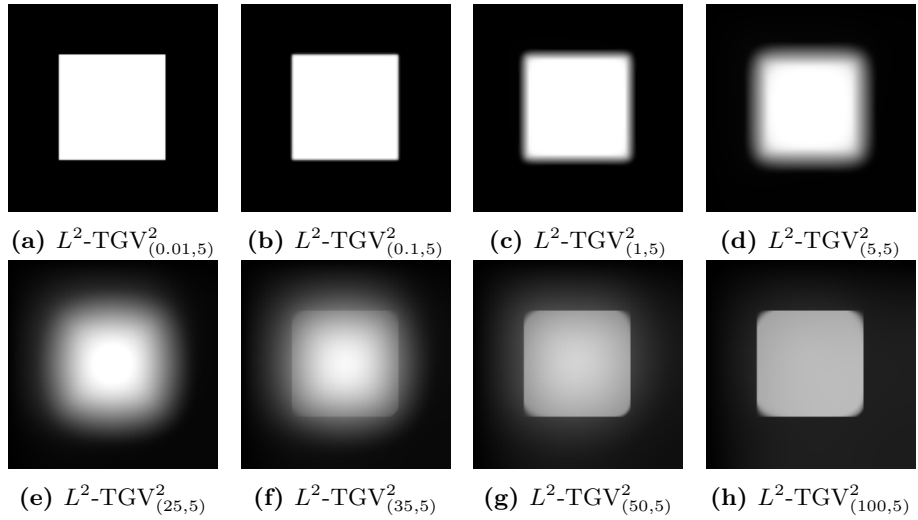}
    \caption{Illustration for $\alpha=5$ of varying $\beta$ parameter for $\TGV^2$
        regularisation with squared $L^2$ fidelity on $f=\chi_{(-32, 32)^2}$.
    }
    \label{fig:square-example-l2-5}
\end{figure}

\section{Conclusion}
\label{sec:conclusion}

In this pair of papers, we have provided a new technique for
studying the jump sets of a general class of regularisation
functionals, not dependent on the co-area formula as existing
results for $\TV$ are. In the case that the fidelity $\phi$ is 
$p$-increasing for $p>1$, besides $\TV$, we proved in Part 1 the
property $\H^{m-1}(J_u \setminus J_f)=0$ for $u$ a solution 
to \eqref{eq:prob} for Huber-regularised total variation.
We also demonstrated in Part 1 that the technique would apply to
non-convex $\TV$ models and the Perona-Malik anisotropic diffusion, 
if these models were well-posed, and had solutions in $\BVspace(\Omega)$.
For variants of $\TGV^2$ using $L^q$, ($q>1$), energies for the
second-order component, we proved that the jump set containment property
holds if the solution $u$ is locally bounded. For $\TGV^2$ itself, we 
obtained much weaker results, depending additionally on the 
differentiability assumptions of Lemma \ref{lemma:bd-w-v-approx} on
the minimising second-order variable $w$.
The two most important further questions that these studies pose are 
whether the assumptions of Lemma \ref{lemma:bd-w-v-approx} on $w$ can be
proved for $\TGV^2$, and whether the local boundedness of the solution
$u$ can be proved for higher-order regularisers in general. In the 
first-order cases this was no work at all.

\section*{Acknowledgements}

This manuscript has been prepared over the course of several years,
exploiting funding from various more short-term projects.
While the author was at the Institute for Mathematics and Scientific 
Computing at the University of Graz, this work was financially supported 
by the SFB research program F32 ``Mathematical Optimization and Applications in 
Biomedical Sciences'' of the Austrian Science Fund (FWF). 
While at the Department of Applied Mathematics and Theoretical Physics
at the University of Cambridge, this work was financially supported by the 
King Abdullah University of Science and Technology (KAUST) Award No.~KUK-I1-007-43 
as well as the EPSRC / Isaac Newton Trust Small Grant ``Non-smooth geometric 
reconstruction for high resolution MRI imaging of fluid transport
in bed reactors'', and the EPSRC first grant Nr.~EP/J009539/1
``Sparse \& Higher-order Image Restoration''.
At the Research Center on Mathematical Modeling (Modemat) at the
Escuela Politécnica Nacional de Quito, this work has been
supported by the Prometeo initiative of the Senescyt.

The author would like to express sincere gratitude to Simon Morgan, 
Antonin Chambolle, Kristian Bredies, and Carola-Bibiane Schönlieb
for fruitful discussions.


\begin{thebibliography}{38}
\providecommand{\natexlab}[1]{#1}
\providecommand{\url}[1]{\texttt{#1}}
\providecommand{\urlprefix}{URL }
\expandafter\ifx\csname urlstyle\endcsname\relax
  \providecommand{\doi}[1]{doi:\discretionary{}{}{}#1}\else
  \providecommand{\doi}{doi:\discretionary{}{}{}\begingroup
  \urlstyle{rm}\Url}\fi
\providecommand{\eprint}[2][]{\url{#2}}

\bibitem[{Ambrosio et~al.(1997)Ambrosio, Coscia and Dal~Maso}]{ambcosdal96}
L.~Ambrosio, A.~Coscia and G.~Dal~Maso, \emph{Fine properties of functions with
  bounded deformation}, Archive for Rational Mechanics and Analysis
  \textbf{139} (1997), 201--238.

\bibitem[{Ambrosio et~al.(2005)Ambrosio, Crippa and
  Maniglia}]{ambrosio-crippa-maniglia}
L.~Ambrosio, G.~Crippa and S.~Maniglia, \emph{Traces and fine properties of a
  {BD} class of vector fields and applications}, Annales de la facult{\'e} des
  sciences de {T}oulouse, S{\'e}r. 6 \textbf{14} (2005), 527--561.

\bibitem[{Ambrosio et~al.(2000)Ambrosio, Fusco and Pallara}]{ambrosio2000fbv}
L.~Ambrosio, N.~Fusco and D.~Pallara, \emph{Functions of Bounded Variation and
  Free Discontinuity Problems}, Oxford University Press, 2000.

\bibitem[{Benning et~al.(2013)Benning, Brune, Burger and
  Müller}]{benning2011higher}
M.~Benning, C.~Brune, M.~Burger and J.~Müller, \emph{Higher-order {TV}
  methods—enhancement via {Bregman} iteration}, Journal of Scientific
  Computing \textbf{54} (2013), 269--310, \doi{10.1007/s10915-012-9650-3}.

\bibitem[{Bertozzi and Greer(2004)}]{bertozzi2004low}
A.~L. Bertozzi and J.~B. Greer, \emph{Low-curvature image simplifiers: Global
  regularity of smooth solutions and laplacian limiting schemes},
  Communications on Pure and Applied Mathematics \textbf{57} (2004), 764--790,
  \doi{10.1002/cpa.20019}.

\bibitem[{Bishop and Goldberg(1980)}]{bishop1980tensor}
R.~L. Bishop and S.~I. Goldberg, \emph{Tensor Analysis on Manifolds}, Dover
  Publications, 1980, {D}over edition.

\bibitem[{Bredies(2013)}]{bredies2014symmetric}
K.~Bredies, \emph{Symmetric tensor fields of bounded deformation}, Annali di
  Matematica Pura ed Applicata \textbf{192} (2013), 815--851,
  \doi{10.1007/s10231-011-0248-4}.

\bibitem[{Bredies and Holler(2013)}]{bredies2013regularization}
K.~Bredies and M.~Holler, \emph{Regularization of linear inverse problems with
  total generalized variation}, SFB-Report 2013-009, University of Graz (2013).

\bibitem[{Bredies et~al.(2011)Bredies, Kunisch and Pock}]{bredies2009tgv}
K.~Bredies, K.~Kunisch and T.~Pock, \emph{Total generalized variation}, SIAM
  Journal on Imaging Sciences \textbf{3} (2011), 492--526,
  \doi{10.1137/090769521}.

\bibitem[{Bredies et~al.(2013)Bredies, Kunisch and Valkonen}]{l1tgv}
K.~Bredies, K.~Kunisch and T.~Valkonen, \emph{Properties of
  {$L^1$-$\mbox{TGV}^2$}: The one-dimensional case}, Journal of Mathematical
  Analysis and Applications \textbf{398} (2013), 438--454,
  \doi{10.1016/j.jmaa.2012.08.053}.
\newline\urlprefix\url{http://math.uni-graz.at/mobis/publications/SFB-Report-2011-006.pdf}

\bibitem[{Bredies and Valkonen(2011)}]{sampta2011tgv}
K.~Bredies and T.~Valkonen, \emph{Inverse problems with second-order total
  generalized variation constraints}, in: \emph{Proceedings of the 9th
  International Conference on Sampling Theory and Applications ({SampTA}) 2011,
  Singapore}, 2011.
\newline\urlprefix\url{\homesiteprefix/SampTA2011.pdf}

\bibitem[{Burger et~al.(2012)Burger, Franek and
  Schönlieb}]{burger2012regularized}
M.~Burger, M.~Franek and C.-B. Schönlieb, \emph{Regularized regression and
  density estimation based on optimal transport}, Applied Mathematics Research
  eXpress \textbf{2012} (2012), 209--253, \doi{10.1093/amrx/abs007}.

\bibitem[{Caselles et~al.(2008)Caselles, Chambolle and
  Novaga}]{caselles2008discontinuity}
V.~Caselles, A.~Chambolle and M.~Novaga, \emph{The discontinuity set of
  solutions of the {TV} denoising problem and some extensions}, Multiscale
  Modeling and Simulation \textbf{6} (2008), 879--894.

\bibitem[{Chambolle and Lions(1997)}]{chambolle97image}
A.~Chambolle and P.-L. Lions, \emph{Image recovery via total variation
  minimization and related problems}, Numerische Mathematik \textbf{76} (1997),
  167--188.

\bibitem[{Chan et~al.(2000)Chan, Marquina and Mulet}]{chan2000high}
T.~Chan, A.~Marquina and P.~Mulet, \emph{High-order total variation-based image
  restoration}, SIAM Journal on Scientific Computation \textbf{22} (2000),
  503--516, \doi{10.1137/S1064827598344169}.

\bibitem[{Chan et~al.(2002)Chan, Kang and Shen}]{chan2002euler}
T.~F. Chan, S.~H. Kang and J.~Shen, \emph{Euler's elastica and curvature-based
  inpainting}, SIAM Journal on Applied Mathematics  (2002), 564--592.

\bibitem[{Ciarlet(2010)}]{ciarlet2010korn}
P.~G. Ciarlet, \emph{On korn’s inequality}, Chinese Annals of Mathematics,
  Series B \textbf{31} (2010), 607--618, \doi{10.1007/s11401-010-0606-3}.

\bibitem[{Conti et~al.(2005)Conti, Faraco and Maggi}]{conti2005new}
S.~Conti, D.~Faraco and F.~Maggi, \emph{A new approach to counterexamples to l
  1 estimates: Korn’s inequality, geometric rigidity, and regularity for
  gradients of separately convex functions}, Archive for Rational Mechanics and
  Analysis \textbf{175} (2005), 287--300, \doi{10.1007/s00205-004-0350-5}.

\bibitem[{Dal~Maso et~al.(2009)Dal~Maso, Fonseca, Leoni and Morini}]{DFLM2009}
G.~Dal~Maso, I.~Fonseca, G.~Leoni and M.~Morini, \emph{A higher order model for
  image restoration: the one-dimensional case}, SIAM Journal on Mathematical
  Analysis \textbf{40} (2009), 2351--2391, \doi{10.1137/070697823}.

\bibitem[{Didas et~al.(2009)Didas, Weickert and Burgeth}]{didas2009properties}
S.~Didas, J.~Weickert and B.~Burgeth, \emph{Properties of higher order
  nonlinear diffusion filtering}, Journal of Mathematical Imaging and Vision
  \textbf{35} (2009), 208--226, \doi{10.1007/s10851-009-0166-x}.

\bibitem[{Duval et~al.(2009)Duval, Aujol and Gousseau}]{duval2009tvl1}
V.~Duval, J.~F. Aujol and Y.~Gousseau, \emph{The {TVL1} model: A geometric
  point of view}, Multiscale Modeling and Simulation \textbf{8} (2009),
  154--189.

\bibitem[{Federer(1969)}]{federer1969gmt}
H.~Federer, \emph{Geometric Measure Theory}, Springer, 1969.

\bibitem[{Hajłasz(1996)}]{hajlasz1996approximate}
P.~Hajłasz, \emph{On approximate differentiability of functions with bounded
  deformation}, Manuscripta Mathematica \textbf{91} (1996), 61--72,
  \doi{10.1007/BF02567939}.

\bibitem[{Hintermüller et~al.(2014)Hintermüller, Valkonen and
  Wu}]{tuomov-tvq}
M.~Hintermüller, T.~Valkonen and T.~Wu, \emph{Limiting aspects of non-convex
  {TV$^q$} models} (2014). In preparation.

\bibitem[{Hinterm\"uller and Wu(2013)}]{HiWu13_siims}
M.~Hinterm\"uller and T.~Wu, \emph{Nonconvex {TV$^q$}-models in image
  restoration: Analysis and a trust-region regularization--based superlinearly
  convergent solver}, SIAM Journal on Imaging Sciences \textbf{6} (2013),
  1385--1415.

\bibitem[{Hinterm\"uller and Wu(2014)}]{HiWu14_coap}
---{}---{}---, \emph{A superlinearly convergent {$R$}-regularized {Newton}
  scheme for variational models with concave sparsity-promoting priors},
  Computational Optimization and Applications \textbf{57} (2014), 1--25.

\bibitem[{Huang and Mumford(1999)}]{huang1999statistics}
J.~Huang and D.~Mumford, \emph{Statistics of natural images and models}, in:
  \emph{IEEE Conference Computer Vision and Pattern Recognition (CVPR)},
  volume~1, IEEE, 1999, volume~1.

\bibitem[{Lellmann et~al.(2014)Lellmann, Lorenz, Sch{\"o}nlieb and
  Valkonen}]{tuomov-krtv}
J.~Lellmann, D.~Lorenz, C.-B. Sch{\"o}nlieb and T.~Valkonen, \emph{Imaging with
  {K}antorovich-{R}ubinstein discrepancy} (2014). Submitted,
  \eprint{1407.0221}.
\newline\urlprefix\url{\homesiteprefix/krtv.pdf}

\bibitem[{Lysaker et~al.(2003)Lysaker, Lundervold and Tai}]{lysaker2003noise}
M.~Lysaker, A.~Lundervold and X.-C. Tai, \emph{Noise removal using fourth-order
  partial differential equation with applications to medical magnetic resonance
  images in space and time}, IEEE Transactions on Image Processing \textbf{12}
  (2003), 1579 -- 1590, \doi{10.1109/TIP.2003.819229}.

\bibitem[{Meyer(2001)}]{meyer2002oscillating}
Y.~Meyer, \emph{Oscillating patterns in image processing and nonlinear
  evolution equations}, American Mathematical Society, Providence, RI, 2001.
  The fifteenth Dean Jacqueline B. Lewis memorial lectures.

\bibitem[{Ochs et~al.(2013)Ochs, Dosovitskiy, Brox and Pock}]{ochsiterated}
P.~Ochs, A.~Dosovitskiy, T.~Brox and T.~Pock, \emph{An iterated l1 algorithm
  for non-smooth non-convex optimization in computer vision}, in: \emph{IEEE
  Conference on Computer Vision and Pattern Recognition (CVPR)}, 2013.

\bibitem[{Papafitsoros and Bredies(2013)}]{papafitsoros2013study}
K.~Papafitsoros and K.~Bredies, \emph{A study of the one dimensional total
  generalised variation regularisation problem} (2013). Preprint.

\bibitem[{Papafitsoros and Sch{\"o}nlieb(2014)}]{papafitsoros2012combined}
K.~Papafitsoros and C.~Sch{\"o}nlieb, \emph{A combined first and second order
  variational approach for image reconstruction}, Journal of Mathematical
  Imaging and Vision \textbf{48} (2014), 308--338,
  \doi{10.1007/s10851-013-0445-4}.

\bibitem[{Shen et~al.(2003)Shen, Kang and Chan}]{shen2003euler}
J.~Shen, S.~Kang and T.~Chan, \emph{Euler's elastica and curvature-based
  inpainting}, SIAM Journal on Applied Mathematics \textbf{63} (2003),
  564--592, \doi{10.1137/S0036139901390088}.

\bibitem[{Temam(1985)}]{temam1985mpp}
R.~Temam, \emph{Mathematical problems in plasticity}, Gauthier-Villars, 1985.

\bibitem[{Valkonen(2014)}]{tuomov-jumpset}
T.~Valkonen, \emph{The jump set under geometric regularisation. {Part 1}: Basic
  technique and first-order denoising} (2014). Submitted, \eprint{1407.1531}.
\newline\urlprefix\url{\homesiteprefix/jumpset.pdf}

\bibitem[{Valkonen et~al.(2013)Valkonen, Bredies and Knoll}]{tuomov-dtireg}
T.~Valkonen, K.~Bredies and F.~Knoll, \emph{Total generalised variation in
  diffusion tensor imaging}, SIAM Journal on Imaging Sciences \textbf{6}
  (2013), 487--525, \doi{10.1137/120867172}.
\newline\urlprefix\url{\homesiteprefix/dtireg.pdf}

\bibitem[{Vese and Osher(2003)}]{vese2003modelingtextures}
L.~A. Vese and S.~J. Osher, \emph{Modeling textures with total variation
  minimization and oscillating patterns in image processing}, Journal of
  Scientific Computing \textbf{19} (2003), 553--572.

\end{thebibliography}
 \providecommand{\homesiteprefix}{http://iki.fi/tuomov/mathematics}
  \providecommand{\eprint}[1]{\href{http://arxiv.org/abs/#1}{arXiv:#1}}


\appendix

\end{document}